\documentclass[11 pt]{article}

\usepackage{setspace}
\usepackage[totalwidth=17cm,totalheight=24cm]{geometry}
\usepackage[T1]{fontenc}
\usepackage[dvips]{graphicx}
\usepackage{epsfig}
\usepackage{amsmath,amssymb,amscd,amsthm}
\usepackage{subfigure}
\usepackage[latin1]{inputenc}
\usepackage{latexsym}
\usepackage{graphics}
\usepackage{colortbl}
\usepackage{color}
\usepackage{ae}
\usepackage{enumitem}
\usepackage{mathabx}
\usepackage{txfonts}

\def\H{{\mathbb H}}

\usepackage{textcomp}
\begin{document}

\setcounter{secnumdepth}{3}
\setcounter{tocdepth}{2}

\newtheorem{definition}{Definition}[section]
\newtheorem{lemma}[definition]{Lemma}
\newtheorem{sublemma}[definition]{Sublemma}
\newtheorem{corollary}[definition]{Corollary}
\newtheorem{proposition}[definition]{Proposition}
\newtheorem{theorem}[definition]{Theorem}

\newtheorem{remark}[definition]{Remark}
\newtheorem{exemple}[definition]{Example}

\newcommand{\mf}{\mathfrak}
\newcommand{\mb}{\mathbb}
\newcommand{\ol}{\overline}
\newcommand{\la}{\langle}
\newcommand{\ra}{\rangle}
\newcommand{\hess}{\mathrm{Hess}}

\newtheorem{thmprime}{Theorem}
\renewcommand{\thethmprime}{1.\arabic{thmprime}\textquotesingle}


\newtheorem{Alphatheorem}{Theorem}
\renewcommand{\theAlphatheorem}{\Alph{Alphatheorem}}

\newtheorem{Alphatheoremprime}{Theorem}
\renewcommand{\theAlphatheoremprime}{\Alph{Alphatheoremprime}\textquotesingle}

\theoremstyle{definition}

\theoremstyle{plain}
\newtheorem*{flp}{Flip algorithm}
\newtheorem{clm}{Claim}

\newcommand{\EM}{\ensuremath}
\newcommand{\norm}[1]{\EM{\left\| #1 \right\|}}

\newcommand{\modul}[1]{\left| #1\right|}

\def\co{\colon\thinspace}
\def\I{{\mathcal I}}
\def\N{{\mathbb N}}
\def\R{{\mathbb R}}
\def\Z{{\mathbb Z}}
\def\Sph{{\mathbb S}}
\def\Tor{{\mathbb T}}
\def\Disk{{\mathbb D}}
\def\Fl{{\mathbb F}}

\def\H{{\mathbb H}}
\def\RP{{\mathbb R}{\mathrm{P}}}
\def\dS{{\mathrm d}{\mathbb{S}}}

\def\phi{\varphi}
\def\epsilon{\varepsilon}
\def\V{{\mathcal V}}
\def\E{{\mathbb E}}
\def\F{{\mathcal F}}
\def\C{{\mathcal C}}
\def\K{{\mathcal K}}

\renewcommand{\span}{\operatorname{span}}
\newcommand{\ang}{\operatorname{ang}}
\newcommand{\area}{\operatorname{area}}
\newcommand{\cone}{\operatorname{cone}}
\newcommand{\pr}{\operatorname{pr}}
\newcommand{\vol}{\operatorname{vol}}
\newcommand{\covol}{\operatorname{covol}}
\newcommand{\st}{\operatorname{st}}
\newcommand{\ost}{\operatorname{st}^\circ}
\newcommand{\inn}{\operatorname{int}}
\renewcommand{\d}{\operatorname{d}}

\def\dist{\mathrm{dist\,}}
\def\diam{\mathrm{diam\,}}

\def\M{{\mathcal M}}
\def\Mt{{\mathcal M}_{\mathrm{tr}}}
\def\T{{\mathcal T}}
\def\Vt{V_{\mathrm{tr}}}

\title{Fuchsian convex bodies: basics of Brunn--Minkowski theory}
\author{Fran\c{c}ois Fillastre\\
University of Cergy-Pontoise\\ UMR CNRS 8088\\ Departement of Mathematics\\
F-95000 Cergy-Pontoise\\ FRANCE\\francois.fillastre@u-cergy.fr}

\date{\today, v2}

\maketitle
\begin{abstract}
The hyperbolic space $ \H^d$ can be defined as a pseudo-sphere in the $(d+1)$ Minkowski space-time.
In this paper, a Fuchsian group $\Gamma$ is a group of linear isometries of the Minkowski space such that 
$\H^d/\Gamma$ is a compact manifold.
We introduce Fuchsian convex bodies, which are closed convex sets in Minkowski space, 
globally invariant for the action of a Fuchsian group.
A volume can be associated to each Fuchsian convex body, and, if the group is fixed, Minkowski addition behaves well.
Then Fuchsian convex bodies can be studied in the same manner as convex bodies of Euclidean space
in
the classical Brunn--Minkowski theory. 
For example, support functions can be defined, as functions on a 
compact hyperbolic manifold instead of the sphere.

The main result is the convexity of the associated volume (it is log concave in the classical setting).
This implies analogs of Alexandrov--Fenchel and Brunn--Minkowski  inequalities. Here the inequalities are reversed.
\end{abstract}

\tableofcontents

\section{Introduction}

There are two main motivations behind the definitions and results presented here. See next section
for a precise definition of Fuchsian convex  bodies, the main object of this paper, and 
Fuchsian convex  surfaces (boundaries of Fuchsian convex  bodies).

The first motivation is to show that
 the geometry of Fuchsian convex  surfaces in the Minkowski space 
is the right analogue of the classical geometry of convex compact hypersurfaces in the Euclidean space.
In the present paper, we show the analogue 
of the basics results of what is called Brunn--Minkowski theory. Roughly speaking, the matter is to study the relations between the sum and 
the volume of the bodies under consideration. Actually here we associate to each convex set 
the volume  of another region of the space, determined by the convex set, so we will call it the \emph{covolume}
of the convex set.
This generalization is as natural as, for example, going from the round sphere to compact hyperbolic surfaces.  
To strengthen this idea, existing results can be put into perspective. Indeed, 
Fuchsian convex surfaces are not new objects. 
As far I know, smooth Fuchsian hypersurfaces appeared in \cite{OS83}, see Subsection~\ref{sub:mink reg}.
The simplest examples of convex Fuchsian surfaces are convex hulls of the orbit of one point for the action of the Fuchsian
group. They were considered in \cite{NP91}, in relation with the seminal papers \cite{pen87,EP88}. See also \cite{CDM97}.
The idea is to study hyperbolic problems via the extrinsic structure given by the Minkowski space. For a recent illustration
see \cite{EGM11}.
The first study of Fuchsian surfaces for their own  is probably \cite{LS00}.
 The authors proved that for any
Riemannian metric on a compact surface of genus $\geq 2$ with negative curvature, there exists an isometric
 convex Fuchsian
surface in the $2+1$-Minkowski space, up to a quotient. In the Euclidean case, the analog problem is known as Weyl problem.
A uniqueness result is also given. 
This kind of result about realization of abstract metrics by (hyper)surfaces
invariant under a group action seems to go back to former papers of F.~Labourie and to \cite{Gro86}.
The polyhedral analog of \cite{LS00} is considered in \cite{Fil11}. An important intermediate result, about 
polyhedral infinitesimal rigidity in $d=2$, was proved in \cite{Sch07} (Fuchsian analogue of Dehn theorem).
More recently, a Fuchsian analogue of the ``Alexandrov prescribed curvature problem''
 was proved in \cite{Ber10}. 
The proof uses optimal mass transport. A refinement of this result in the polyhedral $d=2$ case was obtained in
\cite{isk00}. A solution for  the Christoffel problem (prescribed sum
of the radii of curvature in the regular case) for Fuchsian convex bodies will be given in \cite{fv}
as well as for more general convex sets in the Minkowski space (with or without group action),
similarly to \cite{LLS06}.

The second motivation is that,  up to a quotient, the results presented here are  about
the covolume defined by convex Cauchy surfaces in the simplest case of flat Lorentzian manifolds, namely the quotient
of the interior of the future cone by a Fuchsian group.
It is relevant to consider them in a larger class of flat Lorentzian manifolds, known
as maximal globally hyperbolic Cauchy-compact flat spacetimes. 
They were considered in the seminal paper \cite{Mes07}, see \cite{Mes07+}
and \cite{Bar05,Bon05}. Roughly speaking, one could consider hypersurfaces in the Minkowski space 
invariant under a group of isometries 
whose set of linear isometries forms a Fuchsian group (translations are added). 
In $d=2$, for such smooth strictly convex surfaces, a Minkowski theorem (generalizing Theorem~\ref{thm:alg lin det} in this dimension)  was proved recently in  \cite{BBZ10}.
Maybe some of the basic objects introduced in the present paper could be extended to the point to these manifolds.

The paper is organized as follows. Section~\ref{sec:def} introduces, among main definitions, the tool 
to study (Fuchsian) convex bodies, the support functions. The case of the $C^2_+$ Fuchsian convex bodies
(roughly speaking, the ones with a sufficiently regular boundary) 
is treated in Section~\ref{sec:reg} and the one of polyhedral Fuchsian convex bodies in Section~\ref{sec:pol}. These
two sections are independent. In Section~\ref{sec:gen} the general results are obtained by polyhedral approximation.
It appears that  the proofs of the main results, even though very analogous to the classical ones,  
are simpler than in the Euclidean case.

\subsection*{Acknowledgment}

The author would like to thank Stephanie Alexander, Thierry Barbot, Francesco Bonsante, Bruno Colbois, Ivan Izmestiev, 
Yves Martinez-Maure, Joan Porti, Jean-Marc Schlenker, Graham Smith,
 Rolf Schneider and Abdelghani Zeghib
for attractive discussions about the content of this paper. The author thanks the anonymous referee for his/her comments and suggestions.

Work supported by the ANR GR Analysis-Geometry.

\section{Definitions}\label{sec:def}

\subsection{Fuchsian convex bodies}

The Minkowski space-time of dimension $(d+1)$, $d\geq 1$,  is  
$\mathbb{R}^{d+1}$ endowed with the symmetric bilinear form
$$\langle x,y\rangle_-=x_1y_1+\cdots+x_dy_d-x_{d+1}y_{d+1}.$$
We will denote by $\mathcal{F}$ the interior of the future cone of the origin. It is the set of future time-like 
 vectors: the set of $x$ such that $\langle x,x\rangle_-<0$ (time-like) and the last coordinate of
$x$ for the standard basis is positive (future).
The  pseudo-sphere
contained in $\mathcal{F}$ at distance $t$ from the origin of $\R^{d+1}$ is
$$\H_t^d=\{x\in \R^{d+1}\vert \langle x,x\rangle_-=-t^2, x_{d+1}>0\}. $$ 
All along the paper we identify $\H_1^d$ with the hyperbolic space $\mathbb{H}^d$. 
In particular the isometries of $\H^d$ are identified with the linear isometries of 
the Minkowski space keeping $\H_1^d$ invariant \cite[A.2.4]{BP92}. Note that for any point 
$x\in \mathcal{F}$, there exists $t$ such that $x\in \H_t^d$.
\begin{definition}\label{def: fuchsian body}
A \emph{Fuchsian group} is a subgroup of the  linear isometries group of 
 $\mathbb{R}^{d+1}$,  fixing setwise $\mathcal{F}$ and acting freely cocompactly on  $\mathbb{H}^d$ 
(i.e.~$\H^d/\Gamma$ is a compact manifold).

 A \emph{Fuchsian convex body} is the data of a  convex closed proper subset $K$ of $\mathcal{F}$,
together with a Fuchsian group $\Gamma$, such that $\Gamma K=K$.  
A \emph{$\Gamma$-convex body} is a Fuchsian convex body with Fuchsian group $\Gamma$.

A \emph{Fuchsian convex surface} is the boundary of a Fuchsian convex body.
\end{definition}

A Fuchsian convex body has to be thought as the analogue of
a  convex body (compact convex set), with the compactness condition replaced by 
a ``cocompactness'' condition 
(we will see that a Fuchsian convex body is never bounded). Joan Porti pointed out to the author
that what is done in this paper is probably true without the requirement that the group has no torsion.

We will adapt the classical theory to the Fuchsian case. For that one we mainly follow \cite{Sch93}.

\paragraph{Examples}

The simplest examples of Fuchsian convex surfaces are the $\H_t^d$ 
(note that all Fuchsian groups act freely and cocompactly on $\H_t^d$). Their convex sides  are
Fuchsian
 convex bodies, denoted by $B^d_t$, and $B^d_1$ is sometimes denoted by $B^d$ or $B$. This example shows that a given convex set can be a Fuchsian convex body
for many Fuchsian groups.

Given a Fuchsian group $\Gamma$, we will see in the remaining of the paper two ways of constructing convex Fuchsian
bodies. First, given a finite number of points in $\F$, the convex hull of their orbits for $\Gamma$
is a Fuchsian convex body, see Subsection~\ref{sub:pol}, where a dual construction is introduced.
Second, we will see in Subsection~\ref{sub: reg sup} that any function on the compact hyperbolic manifold
$\H^d/\Gamma$
satisfying a differential relation corresponds to a Fuchsian convex body.
Hence the question of examples reduces to the question of finding the group $\Gamma$, that implies to find compact hyperbolic
manifolds.
Standard concrete examples of compact hyperbolic manifolds can be easily found in the literature
about hyperbolic manifolds. For a general construction in any dimension see \cite{GPS88}. 

Notwithstanding it is not obvious to get explicit generators.
Of course the case $d=1$ is totally trivial as a Fuchsian group is generated by a boost
$\left(
    \begin{array}{cc}
       \cosh t  & \sinh t \\
   \sinh t & \cosh t
    \end{array}
\right),$ for a non-zero real $t$. For $d=2$, explicit generators can be constructed following \cite{mas01}.
For Figure~\ref{fig:polyhedron} and a computation at the end of the paper (the figure comes from 
a part of a Fuchsian convex body that can be manipulate on the author's webpage), the group is the simpliest
acting on $\H^2$, namely the one having a regular octagon as fundamental domain
in a disc model. Generators are given in \cite{kat92}.

\paragraph{Remark on the signature of the bilinear form}

The classical theory of convex bodies uses the usual scalar product on $\R^{d+1}$. Here we used 
the usual bilinear form of signature $(d,1)$. A natural question is to ask what happens 
if we consider a bilinear form of signature $(d+1-k,k)$. (Obviously, the vector structure, 
the volume, the Levi-Civita connection (and hence the geodesics), the topology and the notion of convexity don't depend
on the signature. Moreover, any linear map preserving the bilinear form is of determinant one, hence preserves the volume.)

Let us consider first the case of 
the usual bilinear form with signature $(d-1,2)$ ($d\geq 3$). The set of vectors of pseudo-norm $-1$ is a model of the 
Anti-de Sitter space, which is the Lorentzian analogue of the Hyperbolic space. First of all,
we need groups of linear isometries acting cocompactly on the Anti-de Sitter space. They exist 
only in odd dimensions \cite{BZ04}. Moreover, Anti-de Sitter space does not bound a convex set.

Finally, another interest of the present construction is that, as noted in the introduction, some objects introduced here
could serve to study some kind of flat Lorentzian manifolds (with compact Cauchy surface), which can themselves
be related to some problems coming from General Relativity. It is not clear if 
as many attention is given to pseudo-Riemannian manifolds with different signatures.

\subsection{Support planes}

For  a subset $A$ of $\mathbb{R}^{d+1}$, a \emph{support plane} of $A$ at $x$
is an hyperplane $\mathcal{H}$ with 
$x\in A\cap \mathcal{H}$  and
$A$  entirely contained in one side of $\mathcal{H}$.

\begin{lemma}\label{lem: future convex}
Let $K$ be a $\Gamma$-convex body. Then
\begin{enumerate}[nolistsep,label={\bf(\roman{*})}, ref={\bf(\roman{*})}]
 \item $K$ is  not contained in
a codimension $>0$ plane. \label{nonvide}
\item $K$ is \emph{future convex}:\label{futureconvex}
  \begin{enumerate}[nolistsep,label=(\alph{*}), ref=\ref{futureconvex}\textit{(\alph{*})}]
  \item through each boundary point  there is a support plane; \label{support}
  \item all support planes are space-like;\label{suppspace}
  \item $K$ is contained in the future side of its support planes.\label{future}
  \end{enumerate}
\end{enumerate}
\end{lemma}
\begin{proof}
 By definition $K$ is not empty.
Let $x\in K$. As $K\subset \mathcal{F}$, there exists a $t$ such 
that $x\in \H_t^d$, and by definition, all the elements of the orbit $\Gamma x$ of $x$
belong to $K\cap \H_t^d$. Suppose that $K$ is contained in a codimension 
$>0$ hyperplane
$\mathcal{H}$. Then there would exist a codimension $1$ hyperplane
$\mathcal{H}'$ with $\mathcal{H}\subset \mathcal{H}'$, and $\Gamma x\in \mathcal{H}' \cap \H_t^d $.
This means that  
on $\H_t^d $ (which is homothetic to the hyperbolic space for the induced metric),
 $\Gamma x$ is contained in a totally geodesic hyperplane, a hypersphere or a horosphere (depending
on $\mathcal{H}'$ to be time-like, space-like or light-like), that is clearly impossible. \ref{nonvide} is proved.

\ref{support} is a general property of convex closed subset of $\mathbb{R}^{d+1}$ \cite[1.3.2]{Sch93}.

Let $x\in K$ and let $\mathcal{H}$ be the support plane of $K$ at $x$. There exists $t$ such that $\Gamma x
\subset  \H_t^d$, and all elements of  $\Gamma x$ must be  on one side  of 
$\mathcal{H}\cap \H_t^d$ on $\H_t^d$. Clearly $\mathcal{H}\cap \H_t^d$ can't be a 
totally geodesic hyperplane (of $\H_t^d$), and it can't either be  a horosphere
by Sublemma~\ref{sublem: horo}. 
Hence $\mathcal{H}$ must be space-like, that gives  \ref{suppspace}. The fact that  all elements of  $\Gamma x$ belong to $\H_t^d$
implies that $K$ is in the future side of its support planes, hence \ref{future}.
\end{proof}
\begin{sublemma}\label{sublem: horo}
 Let $\Gamma$ be a group of isometries acting cocompactly on the hyperbolic space $\H^d$.
For any $x\in\H^d$, the orbit $\Gamma x$ meets the interior of any horoball.
\end{sublemma}
\begin{proof}
As the action of $\Gamma$ on $\mathbb{H}^d$ is cocompact, it is well-known that the orbit $\Gamma x$ is discrete and that the Dirichlet regions
for $\Gamma x$
\begin{equation}\label{eq:dirichelt}
D_a(\Gamma)=\{p\in\mathbb{H}^d\vert d(a,p)\leq d(\gamma a,p), 
\forall \gamma\in\Gamma\setminus\{Id\} \}, a\in\Gamma x\end{equation}
where $d$ is the hyperbolic distance,  are bounded \cite{Rat06}.
The sublemma is a characteristic property of discrete sets with bounded Dirichlet regions \cite[Lemma~3]{CDM97}.
\end{proof}

\begin{lemma}\label{lem:conitude}
 Let $K$ be a $\Gamma$-convex body and $x\in K$. For any $\lambda \geq 1$, $\lambda x\in K$.
\end{lemma}
\begin{proof}
 From the definition of $K$, it is not hard to see that it has non empty interior.
And as $K$ 
is closed,  if the lemma was false, there would exist a point on the boundary of $K$ 
and a support plane at this point such that $x$ in its past, that is impossible because of
 Lemma~\ref{lem: future convex}.
\end{proof}

Let us recall the following elementary results, see e.g.~\cite[3.1.1,3.1.2]{Rat06}.
\begin{sublemma}\label{lem:elementary}
 \begin{enumerate}[nolistsep,label={\bf(\roman{*})}, ref={\bf(\roman{*})}]
  \item If $x$ and $y$ are nonzero non space-like vectors in $\mathbb{R}^{d+1}$, 
both past or future, then $\langle x,y\rangle_-\leq 0$ with equality if and only if $x$ and $y$ are
linearly dependent light-like vectors. \label{elem1}
\item If $x$ and $y$ are nonzero non space-like vectors in $\mathbb{R}^{d+1}$, 
both past (resp. future), then the vector $x + y$ is past (resp. future) non space-like.
Moreover $x+y$ is light-like if and only if $x$ and $y$ are linearly dependent
light-like vectors. \label{elem2}
 \end{enumerate}
\end{sublemma}

A future time-like vector $\eta$ orthogonal to a support plane at $x$ of a future convex set $A$ 
is called an \emph{inward normal} of 
$A$ at $x$. This means that $\forall y\in A$, $y-x$ and $\eta$ are two future time-like vectors at the point $x$, then
by Sublemma~\ref{lem:elementary},
$$ \forall y \in A, \langle \eta,y-x\rangle_- \leq 0, \mbox{i.e.~}\langle \eta,y\rangle_- \leq \langle \eta,x \rangle_- $$
or equivalently the sup on all $y\in A$ of  $\langle \eta,y \rangle_-$ is attained at $x$.  
Notice that the set 
$$\{y\in\mathbb{R}^{d+1}\vert \langle y,\eta\rangle_-=\langle x,\eta\rangle_-\} $$
 is the support hyperplane of $A$ at $x$ with inward normal $\eta$.

\begin{lemma}\label{lem: tout vect normal}
 Let $K$ be a $\Gamma$-convex body. For any 
future time-like vector $\eta$,  $\mbox{sup}\{\langle x,\eta\rangle_- \vert x\in K\}$ exists, is attained
at a point of $K$ and is negative. In particular any future time-like vector $\eta$
is an inward normal of $K$.

A future time-like vector $\eta$ is the inward normal of a single support hyperplane of $K$.
\end{lemma}
\begin{proof}
From \ref{elem1} of Lemma~\ref{lem:elementary}, $\{\langle x,\eta\rangle_- \vert x\in K\}$ is bounded
from above
by zero hence the sup exists. 
The sup is a negative number, as a sufficiently small translation of  the vector hyperplane $\mathcal{H}$
orthogonal to $\eta$ in direction of 
$\F$ does not meet $K$. This follows from the separation theorem \cite[1.3.4]{Sch93}, because
the origin is the only common point between $\mathcal{H}$ and the boundary of $\F$.
As $K$ is closed, the sup is attained when the parallel displacement of $\mathcal{H}$  meets $K$.

Suppose that two different support hyperplanes of $K$ have the same inward normal. Hence one is contained in the past of the other, that
is impossible.
\end{proof}

\subsection{Support functions}

Let $K$ be a $\Gamma$-convex body. The \emph{extended support function} 
$H$ of $K$ is 
\begin{equation}\label{def:sup func}
\forall \eta\in\F, H(\eta)=\mbox{sup}\{\langle x,\eta\rangle_- \vert x\in K\}.
\end{equation}
We know from Lemma~\ref{lem: tout vect normal} that it is a negative function on $\F$.
As an example the extended support function of $B_t^d$ is equal to $-t\sqrt{-\langle \eta,\eta\rangle_-}$.

\begin{definition}
 A function $f:A\rightarrow \mathbb{R}$ on a convex subset $A$ of $\R^{d+1}$ is \emph{sublinear} (on $A$) if it is
\emph{positively homogeneous of degree one}:
\begin{equation}\label{def:pos hom}
  \forall \eta\in A, f(\lambda \eta)=\lambda f(\eta)\, \forall \lambda > 0,
\end{equation}
 and \emph{subadditive}:
\begin{equation}
\forall \eta,\mu\in A, f(\eta+\mu)\leq f(\eta)+f(\mu). 
\end{equation}
\end{definition}
A  sublinear function is convex, in particular  it is continuous (by assumptions it takes only finite values in
$A$). (It is usefull to note that for a positively homogeneous of degree one function, convexity and sublinearity are equivalent.)
It is straightforward from the definition that an extended support function is sublinear and $\Gamma$-invariant. 
It is useful to expand the definition of extended support function to the whole space.
The \emph{total support function} of a $\Gamma$-convex body $K$  is 
\begin{equation}\label{eq:ext supp fct}
\forall \eta\in\R^{d+1}, \tilde{H}(\eta)=\mbox{sup}\{\langle x,\eta\rangle_- \vert x\in K\}. 
\end{equation}
We will consider the total support function for any convex subset of $\R^{d+1}$. The infinite value is allowed. 
We have the following important property, see  \cite[Theorem 2.2.8]{Hor07}. 

\begin{proposition}\label{prop:hormander}
Let  $f$ be a lower semi-continuous, convex and positively homogeneous of degree one
function on $\R^{d+1}$ (the infinite value is allowed). 
The set $$  
F=\{x\in\R^{d+1}\vert \langle x,\eta\rangle_- \leq f(\eta) \,\forall \eta\in\R^{d+1} \}
$$
is a closed convex set with total support function $f$.
\end{proposition}

From the definition we get:
\begin{lemma}\label{lem:point}
A convex subset of $\R^{d+1}$ is a point if and only if its total support function is a linear form.
(If the point is $p$, the linear form is
$ \langle \cdot, p\rangle_-$.)

In particular, the total support function of a Fuchsian convex body is never a linear form.
\end{lemma}

The relation between the extended support function and the total support function
is as follows.
\begin{lemma}
The total support function $\tilde{H}$ of a $\Gamma$-convex body with extended support function $H$ is equal to:
 \begin{itemize}[nolistsep]
  \item $H$ on $\F$,
\item $0$ on the future light-like vectors and at $0$,
\item $+\infty$ elsewhere.
 \end{itemize}
Moreover $\tilde{H}$ is a $\Gamma$ invariant sublinear function.
\end{lemma}
\begin{proof}
We have the following cases
\begin{itemize}[nolistsep]
 \item If $\eta$ is future time-like then $\tilde{H}(\eta)=H(\eta)$.
\item If $\eta$ is past time-like or past light-like, then by \ref{elem1} of Sublemma~\ref{lem:elementary}
for $x\in K$, $\langle x,\eta\rangle_->0$, and by Lemma~\ref{lem:conitude}, $\tilde{H}(\eta)=+\infty$.
\item If $\eta$ is space-like, as seen in the proof of \ref{suppspace} of  Lemma~\ref{lem: future convex},
there exists points of $K$ on both side of the orthogonal (for $\langle \cdot,\cdot\rangle_-$)
of $\eta$. Hence there exists $x\in K$ with $\langle x,\eta\rangle_->0$, and by the preceding argument,
 $\tilde{H}(\eta)=+\infty$.
\item If $\eta$ is future light-like, then $\tilde{H}(\eta)=0$. 
 As $\tilde{H}$ is lower semi-continuous (as supremum of a family of continuous
functions)  and as $\tilde{H}=+\infty$ outside of the future cone, this follows from Sublemma~\ref{sub: cl}.
\item By definition, $\tilde{H}(0)=0$.
\end{itemize}
That $\tilde{H}$ is a $\Gamma$ invariant sublinear function follows easily.
\end{proof}

\begin{sublemma}\label{sub: cl}
 Let $H$ be a sublinear function on $\F$ with finite values.
Let us extend it as a convex function on $\R^{d+1}$ by giving the value $+\infty$ outside $\F$.
Let $\tilde{H}$ be the  lower semi-continuous hull of $H$:
$\tilde{H}(x)=\mathrm{liminf}_{x\rightarrow y}H(y)$.

If $H$ is invariant under the action of $\Gamma$, then $H$ is negative or $H\equiv 0$ on $\F$, and $\tilde{H}=0$ on $\partial \F$.
\end{sublemma}
Note that $H\equiv 0$ is the support function of (the closure of) $\F$. 
\begin{proof}
Let $\ell$ be a future light-like vector.
 As $\Gamma$ acts cocompactly on $\H^d$, there exists a sequence of $\gamma_k\in\Gamma$ such that for
any future time-like ray $r$, the sequence $\gamma_n r$  converges to the
ray containing $\ell$ \cite[Example 2, 12.2]{Rat06}. From this sequence we take a sequence $\gamma_k \eta$ for
a future time-like vector $\eta$. We have $\tilde{H}(\gamma_k \eta)=\tilde{H}(\eta)$. From this sequence we
take a sequence of vectors $\eta'_k$ which all have the same $(d+1)$th coordinate $(\eta'_k)_{d+1}$ as $\ell$, say $\ell_{d+1}$
(hence $\eta'_k\rightarrow \ell$). 
We have
$\eta'_k=\ell_{d+1}/(\gamma_k \eta)_{d+1} \gamma_k \eta$, and by homogeneity 
$\tilde{H}(\eta'_k)=\tilde{H}(\ell_{d+1}\eta)/(\gamma_k \eta)_{d+1} $ that goes to $0$ as $k$ goes to infinity ($(\gamma_k \eta)_{d+1}$
goes to infinity).
This proves t$\tilde{H}=0$ on $\partial \F$ as for any $\ell\in\partial\F^*$ and any $\eta\in\F$, 
$\tilde{H}_K(\ell)=\underset{t\downarrow 0}{\mathrm{lim}} H_K(\ell+t(x-\ell)) $ (see for example Theorem~7.5 in \cite{Roc97}).
In the same way we get that $\tilde{H}(0)=0$.

As $\tilde{H}$ is convex and equal to $0$ on $\partial \F$, it is non-positive on $\F$.
Suppose that there exists $x\in\F$ with $\tilde{H}(x)=0$, and let $y\in\F\setminus\{x\}$. 
By homogeneity,  $\tilde{H}( \lambda x)=0$ for all  $\lambda>0$. Up to choose an appropriate
$\lambda$, we can suppose that the line joining $x$ and
$y$ meets $\partial \F$ in two points. Let $\ell$ be the one such that
there exists $\lambda\in]0,1[$ such that $x=\lambda \ell+(1-\lambda) y$. By convexity and because
$\tilde{H}(x)=\tilde{H}(\ell)=0$, we get $0\leq \tilde{H}(y)$, hence $\tilde{H}(y)=0$.
 \end{proof}

\begin{lemma}\label{lem: determination supp fct}
Let $H$ be a  negative sublinear $\Gamma$-invariant
function on $\F$. The set
$$
  K=\{x\in\mathcal{F}\vert \langle x,\eta\rangle_- \leq H(\eta) \,\forall \eta\in\mathcal{F} \}
$$
is a $\Gamma$-convex body with extended support function $H$.
\end{lemma}
\begin{proof} 
Let $\tilde{H}$ be as in Sublemma~\ref{sub: cl}.
From Proposition~\ref{prop:hormander}, the set 
$$
  \tilde{K}=\{x\in\R^{d+1}\vert \langle x,\eta\rangle_- \leq \tilde{H}(\eta) \,\forall \eta\in\R^{d+1} \}
$$
is a closed convex set, with  total support function $\tilde{H}$. Let us see that $\tilde{K}=K$.

As $\tilde{H}(\eta)=+\infty$ outside the closure $\overline{\F}$ of the future cone we have
$$\tilde{K}=\{x\in\R^{d+1}\vert \langle x,\eta\rangle_- \leq \tilde{H}(\eta) \,\forall \eta\in\overline{\F} \}. $$
For $\eta\in\F$, $\tilde{H}(\eta)\leq 0$, it follows that
 $\tilde{K}$
is contained in  $\overline{\F}$:
$$\tilde{K}=\{x\in\overline{\F}\vert \langle x,\eta\rangle_- \leq \tilde{H}(\eta) \,\forall \eta\in\overline{\F} \}. $$
As $H$ is $\Gamma$-invariant, $\tilde{H}$ and $\tilde{K}$ are $\Gamma$-invariant too. 
For $x\in \tilde{K}\cap \partial \F$,
the origin is an accumulating point of $\Gamma x$ from Sublemma~\ref{sub:horo}.
So for any $\eta\in\F$, $\tilde{H}(\eta)$, which is the sup of $\langle x,\eta\rangle_-$ for $x\in \tilde{K}$,
should be zero, that is false. Hence
$$\tilde{K}=\{x\in \F\vert \langle x,\eta\rangle_- \leq \tilde{H}(\eta) \,\forall \eta\in\overline{\F} \}$$
and as $\tilde{H}(\eta)=0$ on $\partial \F$ we get
$$\tilde{K}=\{x\in  \F\vert \langle x,\eta\rangle_- \leq \tilde{H}(\eta) \,\forall \eta\in \F \}=K.$$
The remainder is easy.
\end{proof}
\begin{sublemma}\label{sub:horo}
 Let $\Gamma$ be a Fuchsian group and let $x$ be a future light-like vector. 
Then the origin is an accumulating point of $\Gamma x$.
\end{sublemma}
\begin{proof}
 Suppose it is false. As $\Gamma$ acts cocompactly on $\H^d$,
there exists an horizontal space-like hyperplane $S$ such that a fundamental domain on $\H^{d}$ for the action of
$\Gamma$ lies below $S$. If the origin is not an accumulating point, then there exists $\lambda >0$ such that
the horoball 
$$\{y\in\H^d | -1\leq \langle \lambda x, y \rangle_- <0   \} $$
and its images for the action of $\Gamma$ remain above $S$. This contradicts the definition
of fundamental domain.
\end{proof}

The \emph{polar dual} $K^*$ of a $\Gamma$-convex body $K$ is, if $H$ is the extended support function of $K$:
$$K^*=\{x\in \F | H(x) \leq -1\}. $$
For example, $(B_t^d)^*=B^d_{1/t}$.
It is not hard to see that $K^*$ is a $\Gamma$-convex body, and that $K^{**}=K$ 
(see the convex bodies case \cite[1.6.1]{Sch93}). Moreover the points of the boundary of $K^*$ are 
the $\frac{-1}{H(\eta)}\eta$ for $\eta\in\H^d$. The inverse of this map is
the projection $ f(x)=\frac{x}{\sqrt{-\langle x,x\rangle_-}}$.  Hence, exchanging the roles of $K$ and $K^*$, 
we get that the projection of a Fuchsian convex body along rays from the origin gives is a homeomorphism between $\partial K$ and $\H^d$.

\subsection{Minkowski sum and covolume}

The \emph{(Minkowski) addition} of two sets $A,B\subset \R^{d+1}$ is defined as
$$A+B:=\{a+b | a\in A, b\in B\}. $$
It is well-known that the addition of two convex sets is a convex set. Moreover
 the sum of two future time-like vectors is 
a future time-like vector, in particular it is never zero. So the sum of two $\Gamma$-convex bodies 
is contained in $\F$ and closed \cite[3.12]{RW98}.
As a Fuchsian group $\Gamma$ acts by linear isometries, the sum is a $\Gamma$-convex body,
and the space $\mathcal{K}(\Gamma)$ of $\Gamma$-convex bodies is invariant under the addition. 
Note also that $\mathcal{K}(\Gamma)$ is invariant under multiplication by positive scalars.
It is straightforward to check that extended support functions behave well under these operations:
$$H_{K+L}=H_K+H_L, \, K,L\in \mathcal{K}(\Gamma), $$
$$H_{\lambda K}=\lambda H_K, \,\lambda> 0, K\in\mathcal{K}(\Gamma). $$
Note also that from the definition of the extended support function,
$$K\subset L \Leftrightarrow H_K \leq H_L.$$

Identifying $\Gamma$-convex bodies with
their support functions,  $\mathcal{K}(\Gamma)$ is	 a cone in the vector space of homogeneous of degree $1$, continuous,
real, $\Gamma$-invariant, functions on $\mathcal{F}$. By homogeneity this corresponds to a cone 
in the vector space of continuous
real  $\Gamma$-invariant functions on $\mathbb{H}^d\subset \mathcal{F}$, and to  
 a cone in the vector space of continuous
real  functions on the compact hyperbolic manifold $\mathbb{H}^d/\Gamma$. 
A function in one of this two last cones is called a \emph{support function}.

Let $K\in\mathcal{K}(\Gamma)$. Its \emph{covolume} $\mathrm{covol}(K)$ is the volume of $(\mathcal{F}\setminus K)/\Gamma$
(for the Lebesgue measure of $\R^{d+1}$). It is a finite
positive number and
\begin{equation*}\mathrm{covol}(\lambda K)=\lambda^{d+1}\mathrm{covol}(K).\end{equation*}
Note that
$$K\subset L \Rightarrow \mathrm{covol}(K)\geq \mathrm{covol}(L). $$

As defined above, the covolume of a $\Gamma$-convex body $K$ is the volume of a compact set of 
$\R^{d+1}$, namely the volume of  the intersection of $\F\setminus K$
with a fundamental cone for the action of $\Gamma$. For such compact (non-convex) 
sets there is a Brunn--Minkowski theory, see for example \cite{Gar02}. See also \cite{BE99}.
 But this does not give results about covolume
of $\Gamma$-convex bodies. The reason is that, for two $\Gamma$-convex bodies $K_1$ and $K_2$, 
$\F\setminus (K_1+K_2)$ (from which we define the covolume of $K_1+K_2$) is not equal to $(\F\setminus K_1)+(\F \setminus K_2)$.
For example in $d=1$, $\left(0 \atop 1/2 \right) +\left(5/8\atop 9/8 \right)  \in (\F\setminus B + \F\setminus B)$ 
but does not belong to $\F\setminus(B+B)$.

\section{$C^2_+$ case}\label{sec:reg}

The first subsection is an adaptation of the classical case \cite{Sch93}.
The remainder is the analog of \cite{Ale38} (in \cite{Ale96}). See also 
\cite{BF87}, \cite{Lei93}, \cite{Hor07}, \cite{Bus08}, and \cite{GMT10} for a kind of extension.
 
The objects and results in this section which can be defined intrinsically on a hyperbolic manifold 
are already known in more generality, see
\cite{OS83} and the references therein. See also Subsection~\ref{sub:mink reg}.

\subsection{Regularity of the support function}\label{sub: reg sup}

 \paragraph{Differentiability}
Let $K$ be a $\Gamma$-convex body with extended support function $H$, and let $\eta\in \F$.
From Lemma~\ref{lem: tout vect normal} there exists a unique support hyperplane $\mathcal{H}$
of $K$ with inward normal $\eta$.  
\begin{lemma}
The intersection $F$ of $\mathcal{H}$ and $K$ is reduced to a single point $p$ if and only if $H$ is 
differentiable at $\eta\in\F$.
In this case $p=\nabla_{\eta}H$ (the gradient for $\langle \cdot,\cdot\rangle_-$ of $H$ at $\eta$).  
\end{lemma}
\begin{proof}
 As $H$ is convex all one-sided directional derivatives exist \cite[p.~25]{Sch93}.
Let us denote such derivative in the direction of $u\in\R^{d+1}$ at the point $\eta$ by
$d_{\eta}H(u)$.
The proof of the lemma is based on the following fact:

\emph{The function $\mathbb{R}^{d+1} \ni u\mapsto d_{\eta}H(u)$ is the total support function 
 of $F$.}

Indeed, if $H$ is differentiable at $\eta$, the fact says that the total support function of $F$ is  a linear form, and
from Lemma~\ref{lem:point}, $F$ is a point. Conversely, if $F$ is a point $p$, from  Lemma~\ref{lem:point}
its total support function is a linear form, hence partial derivatives of $H$  exist and as $H$
is convex, this implies differentiability \cite[1.5.6]{Sch93}. Moreover for all $u\in \R^{d+1}$, 
$\langle p,u \rangle_-=d_{\eta} H(u)$.

Now we prove the fact. The function $d_{\eta}H$ is sublinear on
$\R^{d+1}$  \cite[1.5.4]{Sch93}, Proposition~\ref{prop:hormander} applies and
$d_{\eta}H$ is the total support function of
$$F'=\{x\in\R^{d+1}\vert \langle x,u\rangle_- \leq d_{\eta}H(u) \,\forall u\in\R^{d+1} \}.$$
We have to prove that $F'=F$.
Let $\tilde{H}$ be the extension of $H$ to $\R^{d+1}$. By definition of directional derivative, 
the sublinearity of $\tilde{H}$ gives
 $d_{\eta}H\leq \tilde{H}$. From the proof of Lemma~\ref{lem: determination supp fct}, this implies that
 $F'\subset K$.
In particular, for $y\in F'$, $\langle y,\eta \rangle_-\leq H(\eta)$. On the other hand
$y\in F'$ implies $\langle y,-\eta \rangle_-\leq d_{\eta}H(-\eta)=-H(\eta)$ (the last equality 
follows from the definition of directional derivative, using the homogeneity of $H$).
Then $\langle y,\eta \rangle_-=H(\eta)$ so $y\in\mathcal{H}$, hence $F'\subset F=\mathcal{H}\cap K$.

Let $y \in F$. By definition $\langle y, \eta\rangle_- = H(\eta)$ and for any $w\in \F$,
$\langle y,w \rangle_- \leq H(w)$. For sufficiently small positive $\lambda$ and any
$u\in\R^{d+1}$, $w=\eta+\lambda u$ is future time-like  and
$$\langle y,u\rangle_- \leq \frac{H(\eta+\lambda u)-H(\eta)}{\lambda}$$
so when $\lambda\rightarrow 0$ we have $\langle y,u\rangle_-\leq d_{\eta}H(u)$ hence $F\subset F'$. 
The fact is proved.
\end{proof}

If the extended support function $H$ of
a  $\Gamma$-convex body $K$ is differentiable, the above lemma allows to define 
the map
$$\tilde{G}(\eta)=\nabla_{\eta}H$$ 
from $\F$ to $\partial K\subset\mathbb{R}^{d+1}$. This can be expressed in term of $h$, the restriction
of $H$ to $\mathbb{H}^d$. We use ``hyperbolic coordinates'' on $\F$:
an orthonormal frame on
$\H^d$ extended to  an orthonormal frame of $\F$ with the decomposition
$r^2g_{\H^d}-\d r^2$ of the metric on $\F$. $\nabla_{\eta}H$ has $d+1$ entries, and, at $\eta\in\H^d$,
the $d$ first ones are the coordinates of $\nabla_{\eta} h$ (here $\nabla$ is the 
gradient on $\H^d$). We identify $\nabla_{\eta} h\in T_{\eta}\H^d\subset \mathbb{R}^{d+1}$
with a vector of $\mathbb{R}^{d+1}$. The last component of $\nabla_{\eta}H$ is $-\partial H /\partial r(\eta)$, and, using
the homogeneity of $H$, it is equal to $-h(\eta)$ when $\eta\in \mathbb{H}^d$.
Note that at such a point, $T_{\eta}\F$ is the direct sum of $T_{\eta}\mathbb{H}^d$ and
$\eta$. It follows that, for $\eta\in\mathbb{H}^d$,
\begin{equation}\label{eq:nablanabla}\nabla_{\eta}H=\nabla_{\eta}h-h(\eta)\eta. \end{equation}
This has a clear geometric interpretation, see Figure~\ref{fig:nabla}.

\begin{figure}
\centering
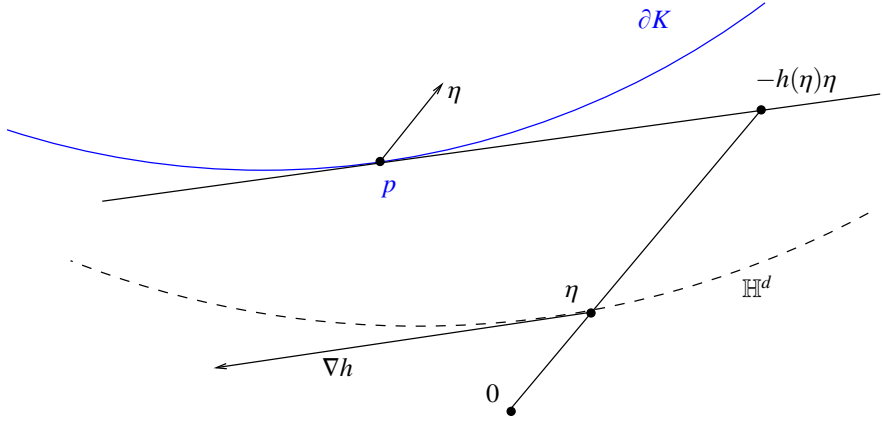  
\caption{Recovering the convex body from its support function in the Minkowski space. 
\label{fig:nabla}} 
  \end{figure}

\paragraph{$C^2$ support function} If the extended support function $H$ is $C^2$, $\tilde{G}$ is $C^1$, and its differential $\tilde{W}$
satisfies $$\langle \tilde{W}_{\eta}(X),Y\rangle_-=D^2_{\eta} H(X,Y).$$

We denote by $G$ the restriction  of $\tilde{G}$ to $\H^d$ and by $W$ its differential (the \emph{reversed shape operator}).
If $T_{\nu}$ is the hyperplane of $\R^{d+1}$ orthogonal to $\nu\in\H^d$ for $\langle \cdot,\cdot,\rangle_-$, $W$ is considered as a map from
$T_{\nu}$ to $T_{\nu}$.
We get from \eqref{eq:nablanabla}, or from the equation above, the Gauss formula and the $1$-homogeneity of $H$,
 using again hyperbolic coordinates on $\F$:
\begin{equation}\label{eq: hess h} W_{ij} = (\nabla^2 h)_{ij}- h \delta_{ij},\end{equation}
with  $\nabla^2$ the second covariant derivative (the  Hessian) on $\H^d$, $\delta_{ij}$  the Kronecker symbol and $h$ the restriction of $H$ to $\H^d$.
In particular $W$ is symmetric, and its real eigenvalues $r_1,\ldots,r_d$  are the \emph{radii of curvature} of $K$.
Taking the trace on both parts of the equation above leads to
\begin{equation}\label{eq:laplacian}
 r_1+\cdots+r_d=\Delta_{\mathbb{H}^d}h-dh
\end{equation}
where $\Delta_{\mathbb{H}^d}$ is the Laplacian on the hyperbolic space. 
It is easy to check that, for $\gamma\in\Gamma$, $\nabla_{\gamma \eta}H=\gamma \nabla_{\eta}H$ and $D^2_{\gamma\eta} H=D^2_{\eta} H$.
In particular the objects introduced above can be  defined on $\H^d/\Gamma$.

\paragraph{$C^2_+$ body} 
Let $K$ be a $\Gamma$-convex body.
 The \emph{Gauss map} $N$ is a multivalued map which associates 
to each $x$ in the boundary of $K$ the set of unit inward normals of $K$ at $x$, which are considered as elements of 
$\mathbb{H}^d$.
If the boundary of $K$ is a $C^2$ hypersurface  and if the Gauss map  is a $C^1$-homeomorphism
from the boundary of $K$ to $\mathbb{H}^d$, $K$ is \emph{$C^2_+$}.
In this case   we can define the \emph{shape operator} $B=\nabla N$, which is a self-adjoint operator. 
Its eigenvalues are the  \emph{principal curvatures} $\kappa_i$ of $K$, and they are never zero
as $B$ has maximal rank by assumption. As $K$ is convex, it is well-known that its principal curvatures are non-negative,
hence they are positive. (This implies that $K$ is actually strictly convex.)
\begin{lemma}\label{lem: supp of regular}
Under the identification of a  $\Gamma$-convex body with its support function, 
the set 
of  $C^2_+$ $\Gamma$-convex body is $C^2_+(\Gamma)$, the set of negative $C^2$ functions $h$ on $M=\mathbb{H}^d/\Gamma$
 such that
\begin{equation}\label{eq:hess supp}
((\nabla^2 h)_{ij}- h \delta_{ij} )>0
\end{equation}
(positive definite) for any orthonormal frame on $M$.
\end{lemma}
It follows that in the $C^2_+$ case 
$G=N^{-1}, W=B^{-1}, \mbox{ and } \displaystyle r_i=\frac{1}{\kappa_i\circ N^{-1}}.$ 
\begin{proof}
Let $K$ be a $C^2_+$ $\Gamma$-convex body,  $h$ its support function and $H$ its extended support function
($h$ is the restriction of $H$ to $\H^d$).
 For any 
$\eta\in\mathbb{H}^d$ we have
\begin{equation}\label{eq: supp et GM}
 h(\eta)=\langle N^{-1}(\eta),\eta\rangle_-,
\end{equation}
and for $\eta\in\F$, introducing the $0$-homogeneous extension $\tilde{N}^{-1}$ of $N^{-1}$ 
we obtain
$$D_{\eta}H(X)=\langle \tilde{N}^{-1}(\eta),X\rangle_-+\langle D_{\eta}\tilde{N}^{-1}(X),\eta\rangle_-, $$
but $D_{\eta}\tilde{N}^{-1}(X)$ belongs to the support hyperplane of $K$ with inward normal $\eta$ so
$D_{\eta}H(X)=\langle \tilde{N}^{-1}(\eta),X\rangle_-. $ 
Hence
$
D^2_{\eta}H(X,Y)=\langle B^{-1}(X),Y\rangle_-, $
in particular 
$H$ is $C^2$, so $h$ is $C^2$ and \eqref{eq:hess supp} is known. 
As $h$ is $\Gamma$-invariant, we get a function of  $C^2_+(\Gamma)$.

Now let $h\in C^2_+(\Gamma)$. We also denote by $h$ the $\Gamma$-invariant map on $\H^d$ which 
projects on $h$, and by $H$
the $1$-homogeneous extension of $h$ to $\F$. The $1$-homogeneity and \eqref{eq:hess supp} imply that
$H$ is convex (in the hyperbolic coordinates, row and column of the Hessian of $H$ corresponding
to the radial direction $r$ are zero), hence  negative sublinear $\Gamma$-invariant, so it is the support function of 
a $\Gamma$-convex body $K$ by Lemma~\ref{lem: determination supp fct}.  
As $h$ is   
$C^2$, we get a  map $G$ from $\H^d$ to $\partial K\subset \R^{d+1}$ which is $C^1$, and regular
from \eqref{eq: hess h} and \eqref{eq:hess supp}. Moreover  $G$ is  surjective by Lemma~\ref{lem: tout vect normal}.
It follows that  $\partial K$ is 
$C^1$. This implies that each point of $\partial K$ has a unique support plane \cite[p.~104]{Sch93}, 
i.e~that the map $G$ is injective. Finally 
it is a $C^1$ homeomorphism.

Let $K^*$ be the polar dual of $K$. We know that the points on the boundary of $K^*$ are graphs above $\H^d$
as they have
the form $\eta/(-h(\eta))$ for $\eta\in\H^d$. Hence $\partial K^*$ is $C^2$ as $h$ is. Moreover 
 the Gauss map image of the point $\eta/(-h(\eta))$ of $\partial K^*$  is
$G(\eta)/\sqrt{-\langle G(\eta),G(\eta)\rangle_-}$:  the Gauss map of $K^*$ is a $C^1$
homeomorphism. It follows that $ K^*$ is   $C^2_+$. In particular its support function is
$C^2$. Repeating the argument, it follows that the boundary of $K^{**}=K$ is $C^2$.
\end{proof}

To simplify the matter in the following, we will restrict ourselves to smooth ($C^{\infty}$) support functions, 
although this restriction will be relevant only in Subsection~\ref{sub:mixed reg}.
We denote by  $C^{\infty}_+(\Gamma)$ the subset of smooth elements of  $C^2_+(\Gamma)$.
It corresponds to $C^{\infty}_+$ $\Gamma$-convex bodies, i.e.~$\Gamma$-convex bodies
with smooth boundary and  with the Gauss map a $C^1$ diffeomorphism (hence smooth).

\begin{lemma}\label{lem: supp reg cone}  $C^{\infty}_+(\Gamma)$   is a convex cone
and $$ C^{\infty}_+(\Gamma) -  C^{\infty}_+(\Gamma)=C^{\infty}(\Gamma)$$
(any smooth function on $\H^d/\Gamma$ is the difference of two functions of $C^{\infty}_+(\Gamma)$).
\end{lemma}
\begin{proof}
It is clear that $ C^{\infty}_+(\Gamma)$ is a convex cone. Let $h_1\in  C^{\infty}_+(\Gamma)$
and $Z\in C^{\infty}(\Gamma)$. As $\H^d/\Gamma$ is compact, for $t$ sufficiently large, 
$Z+th_1$ satisfies \eqref{eq:hess supp} and is a negative function, hence there exists
$h_2\in C^{\infty}_+(\Gamma)$ such that $Z+th_1=h_2$.
\end{proof}

\subsection{Covolume and Gaussian curvature operator}

Let $K$ be a $C^2_+$ $\Gamma$-convex body and let
$P(K)$ be $\F$ minus the interior of $K$. 
As $P(K)/\Gamma$ is compact, the divergence theorem gives
$$ \int_{P(K)/\Gamma} \mbox{div}X \d P(K)= -\int_{\partial K/\Gamma} \langle X,\eta\rangle_- \d \partial K,$$
where $\eta$ is the unit outward  normal of $\partial K/\Gamma$ in $P(K)/\Gamma$ (hence
it corresponds in the universal cover to the unit inward normal of $K$).
If $X$ is the position vector in $\F$ we get 
$$
(d+1)\mathrm{covol}(K)=-\int_{\partial K/\Gamma}  h\circ N \d \partial K
$$
with $h$  the support function of $K$  and $N$ the Gauss map.

The 
\emph{Gaussian curvature} (or Gauss--Kronecker curvature) $\kappa$ of $K$ is the product of the principal curvatures.
We will consider the map $\kappa^{-1}$ which associates to each $h\in C^{\infty}_+(\Gamma)$ the inverse of the Gaussian curvature
of the convex body supported by $h$: 
\begin{equation}\label{eq: def kappa}
\kappa^{-1}(h)=\prod_{i=1}^dr_i(h)\stackrel{\eqref{eq: hess h}}{=}\det \left((\nabla^2 h)_{ij}- h \delta_{ij}\right).
\end{equation}

As the curvature is the Jacobian of the Gauss map, we get
$$(d+1)\mathrm{covol}(K)=-\int_{M} h \kappa^{-1}(h) \d M$$
where $\d M$ is the volume form on $M=\H^d/\Gamma$.
Finally let us consider the covolume as a functional on $C^{\infty}_+(\Gamma)$, which 
extension to the whole $C^{\infty}(\Gamma)$ is immediate:
\begin{equation}\label{eq: vol reg}\mathrm{covol}(X)=-\frac{1}{d+1}\lgroup X,\kappa^{-1}(X) \rgroup,  X\in C^{\infty}(\Gamma)\end{equation}
with $\lgroup\cdot,\cdot \rgroup$ the scalar product on $L^2(M)$.

We will consider  $C^{\infty}(\Gamma)$ as a Fr\'echet space with the usual seminorms
$$\|f\|_n=\sum_{i=1}^n \sup_{x\in M} |\nabla^i f(x) |, $$
with $\nabla^i$ the $i$-th covariant derivative and $|\cdot|$ the norm, both given by the Riemannian metric of $M$.
All derivatives will be  directional (or G\^ateaux) derivatives in Fr\'echet spaces  as in \cite{Ham82}:
\begin{equation}\label{eq: der dir}
 D_{Y}\mathrm{covol}(X)=\lim_{t\rightarrow 0}\frac{\mathrm{covol}(Y+tX)-\mathrm{covol}(Y)}{t}, X,Y\in C^{\infty}(\Gamma).
\end{equation}

\begin{lemma} 
The function $\mathrm{covol}$ is $C^{\infty}$  on  $C^{\infty}(\Gamma)$, and for $h\in C^{\infty}_+(\Gamma), X,Y\in C^{\infty}(\Gamma)$,
we have:
\begin{eqnarray}
\ D_{h}\mathrm{covol}(X)=-\lgroup X, \kappa^{-1}(h) \rgroup \label{eq: der vol reg}, \\ 
\ D_{h}^2 \mathrm{covol} (X,Y)=-\lgroup X,D_{h}\kappa^{-1}(Y)\rgroup. \label{eq: der sec vol reg}
\end{eqnarray}
Moreover \eqref{eq: der vol reg} is equivalent to
\begin{equation}\label{eq: k self adj}
 \lgroup X, D_h\kappa^{-1} (Y)\rgroup= \lgroup Y, D_h\kappa^{-1} (X)\rgroup.
\end{equation}
\end{lemma}
\begin{proof}

The second order differential operator  $\kappa^{-1}$ is smooth as the determinant is smooth \cite[3.6.6]{Ham82}.
Differentiating \eqref{eq: vol reg} we get
\begin{equation}\label{eq: der vol reg rough}
 D_{h}\mathrm{covol}(X)=-\frac{1}{d+1}\left(\lgroup X,\kappa^{-1}(h) \rgroup+\lgroup h,D_h\kappa^{-1}(X) \rgroup\right),
\end{equation}
but the bilinear form $\lgroup \cdot,\cdot\rgroup$ is continuous for the seminorms $\|\cdot \|_n$
(recall that it suffices to check continuity in each variable \cite[2.17]{Rud91}). It follows that $\mathrm{covol}$ is $C^1$,
and by iteration that it is $C^{\infty}$.

If \eqref{eq: der vol reg} is true we get \eqref{eq: der sec vol reg}, and this expression is symmetric as
$\mathrm{covol}$ is $C^2$, so \eqref{eq: k self adj} holds.  

Let us suppose that \eqref{eq: k self adj} is true. From  \eqref{eq: def kappa},
 $\kappa^{-1}$ is homogeneous of degree $d$, 
that gives $D_h\kappa^{-1} (h)=d\kappa^{-1}(h)$. Using this
in \eqref{eq: k self adj} with $Y=h$ gives
$$d\lgroup X, \kappa^{-1} (h)\rgroup= \lgroup h, D_h\kappa^{-1} (X)\rgroup.$$
Inserting this
equation in \eqref{eq: der vol reg rough} leads to \eqref{eq: der vol reg}.

A proof of \eqref{eq: k self adj} is done in \cite{CY76} (for the case of $C^2$ functions on the sphere).
See also \cite{OS83} and reference therein for more generality.
We will prove \eqref{eq: der vol reg} following \cite{Hor07}.
From the definition of $\kappa^{-1}$, the map $D_h\kappa^{-1}(\cdot)$ is
linear, hence from \eqref{eq: der vol reg rough} $D_h\mathrm{covol}(\cdot)$ is also linear, so by Lemma~\ref{lem: supp reg cone}
it suffices to prove \eqref{eq: der vol reg} for $X=h'\in C^{\infty}_+(\Gamma)$.
 We denote by 
$K$ (resp. $K'$) the $\Gamma$-convex body supported by $h$ (resp. $h'$) and by
$N$ (resp. $N'$) its Gauss map. We have, for $\eta\in \F, \epsilon>0,$
$$ h(\eta)+\epsilon h'(\eta)=\langle \eta, N^{-1}(\eta)+\epsilon (N')^{-1}(\eta)\rangle_-$$
i.e~$h+\epsilon h'$ supports the hypersurface with position vector $N^{-1}(\eta)+\epsilon (N')^{-1}(\eta)$.

For a compact $U\subset \R^d$, if $f : U\rightarrow\R^{d+1} $ is a local parametrization of $\partial K$,
let us introduce
$$F : U\times [0,\epsilon]\rightarrow \R^{d+1},
(y,t)\mapsto f(y)+t (N')^{-1}(N(f(y))).$$
It is a local parametrization of the set between the boundary of $K$ and
the boundary of $K+ \epsilon K'$. 
Locally, its covolume (which corresponds to $\mathrm{covol}(h+\epsilon h')-\mathrm{covol}(h)$) is computed as
\begin{equation}
\label{eq:loc vol}\int_{F(U\times [0,\epsilon])} \d \vol=\int_0^{\epsilon }\int_U | \mbox{Jac} F |\d y \d t.
\end{equation}
The Jacobian of $F$ is equal to 
$\left((N')^{-1}(N(f(y))), \frac{\partial f}{\partial y_1},\ldots, \frac{\partial f}{\partial y_d}\right)
+ t R$ where $R$ is a remaining term, and its determinant is equal to the determinant of 
$\left((N')^{-1}(N(f(y))), \frac{\partial f}{\partial y_1},\ldots, \frac{\partial f}{\partial y_d}\right)$
plus $t$ times remaining terms. As $(\frac{\partial f}{\partial y_1},\ldots, \frac{\partial f}{\partial y_d})$ 
form a basis of the tangent hyperplane of $\partial K$, and as $N$ is normal to this hyperplane,
the determinant is equal
to  $\langle  (N')^{-1}(N(f(y))), N(f(y))\rangle_-=h'(N(f(y)))$ times $|\mbox{Jac} f|$, plus $t$ times remaining terms.
The limit of \eqref{eq:loc vol} divided by $\epsilon$ when $\epsilon \rightarrow 0$ gives
$$\int_U h'(N(f(y)) |\mbox{Jac} f | \d y=\int_{f(U)} h'(N) \d\partial K.$$ 
The result follows by decomposing the boundary of $K$ with suitable coordinate patches.
\end{proof}

The main result of this section is the following.
\begin{theorem}\label{thm: vol reg conv}
The second derivative of $\mathrm{covol}: C^{\infty}(\Gamma) \rightarrow \R$ is positive definite. In particular
the covolume of $C^{\infty}_+$ $\Gamma$-convex bodies 
is strictly convex.
\end{theorem}

Let us have a look at the case $d=1$. In this case $\kappa^{-1}=r$, the unique radius of curvature.
We parametrize the branch of the hyperbola by $(\sinh t,\cosh t)$, and  $h$ becomes a function from $\R$ to $\R_-$. Then
\eqref{eq:laplacian} reads
$$\kappa^{-1}(h)(t)=-h(t)+h''(t),$$
and, as $h$ is $\Gamma$-invariant, we can consider $\kappa^{-1}$ as a linear operator on the set of $C^{\infty}$ functions on 
$[0,\ell]$, 
if $\ell$  is the length of the circle $\H^1/\Gamma$.  Using integration by parts and the fact that 
$h$ is $\ell$-periodic, we get
$$D_h^2\mathrm{covol}(h,h)=-\lgroup h,\kappa^{-1}(h)\rgroup=-\int_0^{\ell} h \kappa^{-1}(h)= \int_0^{\ell}(h^2+h'^2).$$

We will prove a more general version of Theorem~\ref{thm: vol reg conv} in the next section, using the theory of mixed-volume. 
The proof is based on the 
following particular case.

\begin{lemma}\label{lem: vol def pos sphere}
 Let $h_0$ be the  support function of $B^d$ 
(i.e.~$h_0(\eta)=-1$).
Then $D_{h_0}^2 \mathrm{covol}$ is positive definite. 
\end{lemma}
\begin{proof}
Let $X\in C^{\infty}(\Gamma)$. 
From the definition \eqref{eq: def kappa}  of $\kappa^{-1}$ 
$$D_h \kappa^{-1} (X)=\kappa^{-1}(h) \sum_{i=1}^d r_i^{-1}(h)D_hr_i(X)$$
and as $r_i(h_0)=1$,  
$$D_{h_0} \kappa^{-1} (X)=\sum_{i=1}^d D_{h_0}r_i(X).$$
Differentiating \eqref{eq:laplacian} on both side at $h_0$ and passing to the quotient, the equation above gives
$$D_{h_0} \kappa^{-1} (X)=-dX+\Delta_{M}X,$$
where $\Delta_{M}$ is the Laplacian on $M=\H^d/\Gamma$. From \eqref{eq: der sec vol reg},
$$D^2_{h_0}\mathrm{covol}(X,X)=d\lgroup X,X \rgroup-\lgroup\Delta_{M}X,X\rgroup, $$
which is positive by property of the Laplacian, as $M=\H^d/\Gamma$ is compact. 
\end{proof}

\subsection{Smooth Minkowski Theorem}\label{sub:mink reg}

One can ask if, given a positive function $f$ on a hyperbolic compact manifold $M=\H^d/\Gamma$, 
it is the Gauss curvature of a  $C^{2}_+$ convex Fuchsian surface and if the former one is unique.
By Lemma~\ref{lem: supp of regular} and definition of the Gauss curvature, the question reduces 
to  know if there exists a (unique) function 
$h$ on $M$ such that, in an orthogonal frame on $M$,
$$f=\det((\nabla^2h)_{ij}-h\delta_{ij}) $$
and
$$((\nabla^2h)_{ij}-h\delta_{ij})>0.$$
This PDE problem is solved in \cite{OS83} in the smooth case. Their main result (Theorem~3.4) can be written as
follows.
\begin{theorem}\label{thm: reg mink thm}
 Let $\Gamma$ be a Fuchsian group, 
 $f:\H^d\rightarrow \R_+$ be a positive $C^{\infty}$ $\Gamma$-invariant function.

 There exists a unique  $C^{\infty}_+$ $\Gamma$-convex  body with Gauss curvature $f$.
\end{theorem}

\subsection{Mixed curvature and mixed-covolume}\label{sub:mixed reg}

The determinant is a homogeneous polynomial of degree $d$, and we denote 
by $\det( \cdot,\ldots,\cdot)$ its polar form, that is
the unique symmetric $d$-linear form such that
$$\det( A,\ldots,A)=\det(A) $$
for any $d\times d$ symmetric matrix $A$ (see for example Appendix~A in \cite{Hor07}). We will need the following key result.

\begin{theorem}[{\cite[p.~125]{Ale96}}]\label{thm:alg lin det}
 Let $A,A_3\ldots,A_d$ be positive definite  $d\times d$ matrices and $Z$ be a symmetric matrix.
Then
 $$\det (Z,A,A_3,\ldots,A_d)=0\Rightarrow \det(Z,Z,A_3,\ldots,A_d)\leq 0,$$
and equality holds if and only if $Z$ is identically zero.
\end{theorem}

For any orthonormal frame on $M=\mathbb{H}^d/\Gamma$ and
for $X_k\in C^{\infty}(\Gamma)$, let us denote $$X_k'':= (\nabla^2 X_k)_{ij}- X_k \delta_{ij}$$
and let us introduce the \emph{mixed curvature}
$$\kappa^{-1}(X_1,\ldots,X_d):=\det(X_1'',\ldots,X_d''). $$ 

As $\mathrm{covol}(X)=-\frac{1}{d+1}\lgroup X, \kappa^{-1}(X)\rgroup$,
 $\mathrm{covol}$ is a homogeneous polynomial of degree $d+1$. Its polar form
$\mathrm{covol}(\cdot,\ldots,\cdot)$ ($(d+1)$ entries) is the \emph{mixed-covolume}. 

\begin{lemma}\label{eq: reg mix vol gen}
We have the following equalities, for $X_i\in C^{\infty}(\Gamma)$.
\begin{enumerate}[nolistsep,label={\bf(\roman{*})}, ref={\bf(\roman{*})}]
 \item $D^{d-1}_{X_2}\kappa^{-1}(X_3,\ldots,X_{d+1})=d! \kappa^{-1}(X_2,\ldots,X_{d+1})$,\label{der kappa}
\item $D_{X_1}\mathrm{covol} (X_2)=(d+1)\mathrm{covol}(X_2,X_1,\ldots,X_1)$,\label{der mixed reg1}
\item $D^2_{X_1}\mathrm{covol} (X_2,X_3)=(d+1)d\mathrm{covol}(X_2,X_3,X_1,\ldots,X_1)$, \label{der mixed reg2}
\item $D^{d}_{X_1} \mathrm{covol} (X_2,\ldots,X_{d+1})=(d+1)!\mathrm{covol}(X_1,\ldots,X_{d+1})$,\label{der mixed reg}
\item $\mathrm{covol}(X_1,\ldots,X_{d+1})=-\frac{1}{d+1}\lgroup X_1, \kappa^{-1}(X_2,\ldots,X_{d+1})\rgroup$.\label{mv reg}
\end{enumerate}
\end{lemma}
\begin{proof}
\ref{der kappa} and \ref{der mixed reg} are proved by induction on the order of the derivative,
using the definition of directional derivative and the expansion of the multilinear forms.
\ref{der mixed reg1} and \ref{der mixed reg2} are obtained by the way.
\ref{mv reg} follows from  \eqref{eq: der vol reg}, \ref{der kappa} and \ref{der mixed reg}.
\end{proof}
\begin{corollary}\label{cor: reg mix vol pos}
 For $h_i\in C^{\infty}_+(\Gamma)$, $\mathrm{covol}(h_1,\ldots,h_{d+1})$ is positive.
\end{corollary}
\begin{proof}
As $h_i\in C^{\infty}_+(\Gamma)$, $h_i''$ is positive definite, hence 
$\kappa^{-1}(h_2,\ldots,h_{d+1})>0$ \cite[(5) p.~122]{Ale96}.
The result follows from \ref{mv reg} because $h_1<0$.
\end{proof}

Due to \ref{der mixed reg2} of the preceding lemma, the following result implies Theorem~\ref{thm: vol reg conv}.

\begin{theorem}\label{thm:hess vol def pos reg}
 For any $h_1,\ldots,h_{d-1}$ in $C^{\infty}_+(\Gamma)$, the symmetric bilinear form on $(C^{\infty}(\Gamma))^2$
$$\mathrm{covol}(\cdot,\cdot,h_1,\ldots,h_{d-1}) $$
is positive definite.
\end{theorem}
\begin{proof}
We use a continuity method. We consider the paths
$h_i(t)=th_i+(1-t)h_0$, $i=1,\ldots,d-1$, $t\in[0,1]$, where $h_0$ is the (quotient of the) support function of 
$B^d$ and we
denote $$\mathrm{covol}_t(\cdot,\cdot):=\mathrm{covol}(\cdot,\cdot,h_1(t),\ldots,h_{d-1}(t)).$$
The result follows from the facts:
\begin{enumerate}[nolistsep,label={\bf(\roman{*})}, ref={\bf(\roman{*})}]
 \item $\mathrm{covol}_0$ is positive definite,\label{preuve1}
\item if, for each $t_0\in[0,1]$,  $\mathrm{covol}_{t_0}$ is positive definite, then 
$\mathrm{covol}_t$ is positive definite for $t$ near $t_0$,\label{preuve2}
\item if $t_n \in [0,1]$ with $t_n\rightarrow t_0$ and $\mathrm{covol}_{t_n}$ is positive definite, then
$\mathrm{covol}_{t_0}$ is positive definite.  \label{preuve3}
\end{enumerate}
\ref{preuve1} is  Lemma~\ref{lem: vol def pos sphere}. Let $t_0$ as in \ref{preuve2}.
By Lemma~\ref{lem:elliptic}, each $\kappa^{-1}(\cdot,h_1(t),\ldots,h_{d-1}(t))$
inherits standard properties of elliptic self-adjoint operators on compact manifolds 
(see for example \cite{Nic07}), and we can apply \cite[Theorem 3.9 p.~392]{Kat95}: as the deformation of the operators 
is polynomial in $t$, the eigenvalues change analytically with $t$, for $t$ near $t_0$.
In particular if $t$ is sufficiently close to $t_0$, the eigenvalues remain positive and \ref{preuve2} holds.

Let $t_n$ be as in \ref{preuve3}. For any non zero $X\in C^{\infty}(\Gamma)$
we have $\mathrm{covol}_{t_n}(X,X)>0$ with
$$\mathrm{covol}_{t_n}(X,X)=\int_M X \kappa^{-1}(X, (1-t_n)h_0+t_nh_1,\ldots,  (1-t_n)h_0+t_nh_{d-1}) \d M.$$
As $\kappa^{-1}$ is multilinear and as $t_n<1$, it is easy to see that the function in the integrand above is
bounded by a function (of the kind $X\sum \vert \kappa^{-1}(X,*,\ldots,*)\vert$
where each $*$ is $h_0$ or a $h_i$) 
which does not depend on $n$ and is continuous on the compact $M$. 
By Lebesgue's dominated convergence theorem, $\mathrm{covol}_{t_0}(X,X)\geq 0$, and 
by Lemma~\ref{lem: noyau trivial reg}   $\mathrm{covol}_{t_0}(X,X)>0$, and \ref{preuve3} is proved.
\end{proof}

\begin{lemma}\label{lem:elliptic}
 For any $h_1,\ldots,h_{d-1}$ in $C^{\infty}_+(\Gamma)$, the operator 
$\kappa^{-1}(\cdot,h_1,\ldots,h_{d-1})$ is formally self-adjoint linear second order elliptic. 
\end{lemma}
\begin{proof}
It is formally self-adjoint because of the symmetry of the mixed-covolume.  It is clearly second order linear.
Let $Z\in C^{\infty}(\Gamma)$. 
From properties of the mixed determinant  \cite[p.~121]{Ale96}, $\kappa^{-1}(Z,h_1,\ldots,h_{d-1})$ can be written,
for an orthonormal frame on $M$,
$$\sum_{i,j=1}^d \det(h_1'',\ldots,h_{d-1}'')_{ij}\left( (\nabla^2 Z)_{ij}-Z\delta_{ij}\right) $$
where $\det(h_1'',\ldots,h_{d-1}'')_{ij}$ is, up to a constant factor, the mixed determinant of the matrices  obtained
from the $h_k''$ by deleting the $i$th row and the $j$th column. Let us consider local coordinates on $M$ around a point $p$
such that at $p$, $\kappa^{-1}(Z,h_1,\ldots,h_{d-1})$ has the expression above.
By definition of $C^{\infty}_+(\Gamma)$, $h_k''$ are positive definite  at $p$ and then at $p$
$$\sum_{i,j=1}^d \det(h_1'',\ldots,h_{d-1}'')_{ij} x_ix_j $$
is positive definite  \cite[Lemma~II p.~124]{Ale96}.
\end{proof}
\begin{lemma}\label{lem: noyau trivial reg}
 For any $h_1,\ldots,h_{d-1}$ in $C^{\infty}_+(\Gamma)$, the symmetric bilinear form
$$\mathrm{covol}(\cdot,\cdot,h_1,\ldots,h_{d-1}) $$
has trivial kernel.
\end{lemma}
\begin{proof}
 Suppose that $Z$ belongs to the kernel of $\mathrm{covol}(\cdot,\cdot,h_1,\ldots,h_{d-1}) $. As 
$\lgroup\cdot,\cdot\rgroup$ is an inner product, $Z$ belongs to the kernel of 
$\kappa^{-1}(\cdot,h_1,\ldots,h_{d-1})$:
$$\det(Z'',h_1'',\ldots,h_{d-1}'')=0. $$
 As $h_i''$ are positive definite matrices, by definition of $C^{\infty}_+(\Gamma)$, Theorem~\ref{thm:alg lin det}
implies that $$\det(Z'',Z'',h_2'',\ldots,h_{d-1}'')\leq 0 $$
so 
$$0=\mathrm{covol}(Z,Z,h_1,\ldots,h_{d-1})=-\int_M h_1 \kappa^{-1}(Z,Z,h_2,\ldots,h_{d-1})\leq 0  $$
but $h_1 <0$ hence
$$\det(Z'',Z'',h_2'',\ldots,h_{d-1}'')= 0, $$
and Theorem~\ref{thm:alg lin det} says that $Z''=0$.
Consider the $1$-homogeneous extension $\tilde{Z}$ of the $\Gamma$ invariant map
on $\H^d$ defined by $Z$.  
From Subsection~\ref{sub: reg sup} it follows that 
the Hessian of $\tilde{Z}$ in $\F$ is zero, hence that 
$\tilde{Z}$ is affine. By invariance $\tilde{Z}$ must be constant, and by homogeneity $\tilde{Z}=0$ hence $Z=0$.
\end{proof}

\paragraph{Remark on Fuchsian Hedgehogs}

If we apply Cauchy--Schwarz inequality to the inner product of
Theorem~\ref{thm:hess vol def pos reg}, we get a ``reversed Alexandrov--Fenchel inequality''
(see Theorem~\ref{thm:general}) for
$C^{\infty}_+$ convex bodies, but also for any smooth function $h$ on the hyperbolic manifold
$\H^d/\Gamma$. From Lemma~\ref{lem: supp reg cone} there exist
two elements $h_1,h_2$ of $C^{\infty}_+(\Gamma)$   with $h=h_1-h_2$. Hence
$h$ can be seen as the ``support function'' of the (maybe non convex) hypersurface 
made of the points $\nabla_{\eta}(H_1-H_2), \eta\in\F$.
For example if $h_1$ and $h_2$ are the support functions of respectively
$B_{t_1}$ and $B_{t_2}$, then $h$ is the support function of a pseudo-sphere in $\F$ if
$t_1-t_2> 0$, of a point (the origin) if $t_1-t_2=0$ and of a pseudo-sphere in the past
cone if $t_1-t_2<0$.

More generally, we could introduce ``Fuchsian hedgehogs'',
whose ``support functions'' are difference of  support functions of two $\Gamma$-convex bodies. 
 They form the vector space in which the support functions of $\Gamma$-convex bodies naturally live.
In the Euclidean space, they were introduced in \cite{LLR88}. An Euclidean analog of the reversed Alexandrov--Fenchel inequality 
for smooth Fuchsian  hedgehogs described above is done in \cite{MM99}, among other results.
It would be interesting to know if other results about hedgehogs have a Fuchsian analogue.

\section{Polyhedral case}\label{sec:pol}

The classical analogue of this section comes from \cite{Ale37} (see \cite{Ale96}). See also
\cite{Sch93} and \cite{Ale05}. The toy example $d=1$ is considered in the note \cite{polymink}.

\subsection{Support vectors}\label{sub:pol}

\paragraph{Definition of Fuchsian convex polyhedron}

 The notation 
$a^{\bot}$ will represent the affine hyperplane over the vector hyperplane orthogonal to the
vector $a$ and passing through $a$:
\begin{equation}\label{eq def affine}
a^{\bot}=\{x\in\R^{d+1}| \langle x,a\rangle_-=\langle a,a\rangle_- \}. 
\end{equation}
 
\begin{definition}
Let $R=(\eta_1,\ldots,\eta_n)$, $n\geq 1$, with $\eta_i$ (pairwise non-collinear) vectors in the future cone $\mathcal{F}$, 
and let $\Gamma$ be a Fuchsian group.
 A \emph{$\Gamma$-convex  polyhedron} is the boundary of the intersection of the half-spaces 
bounded by the hyperplanes
$$(\gamma \eta_i)^{\bot}, \forall \gamma\in\Gamma, \forall i=1,\ldots,n,$$
 such that the vectors $\eta_i$ are inward pointing. 
\end{definition}

See Figure~\ref{fig:polyhedron} for a simple example.

\begin{figure}
\centering
\includegraphics[scale=0.4]{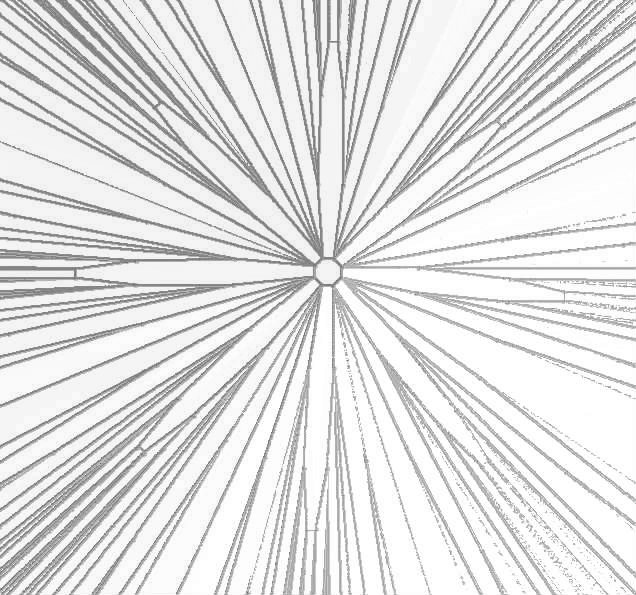}  
\caption{A piece of a $\Gamma$-convex polyhedron in $d=2$ seen from the bottom.
It is made with the orbit of $(0,0,1)$ for the Fuchsian group having a regular octagon as fundamental domain in $\H^2$.
\label{fig:polyhedron}} 
  \end{figure}

\begin{lemma}
A $\Gamma$-convex polyhedron $P$ 
\begin{enumerate}[nolistsep,label={\bf(\roman{*})}, ref={\bf(\roman{*})}]
\item is a $\Gamma$-convex body, \label{basic pol 1}
\item has a countable number of facets,
\item is locally finite,
\item  each face is a  convex  Euclidean polytope. 
\end{enumerate}
\end{lemma}
Here convex polytope means convex compact polyhedron.
\begin{proof}
We denote by $P_i$ the $\Gamma$-convex polyhedron
made from the vector $\eta_i$ and the group $\Gamma$. 
We will prove the lemma for $P_i$. The general case follows 
because $P$ is the intersection of a finite number of $P_i$.
 All the elements of $\Gamma \eta_i$ belong to  $\H_{t_i}^d$,
 on which $\Gamma$ acts cocompactly. Up to a  homothety, it is more suitable to consider that $\H_{t_i}^d$ is
$\H_1^d=\mathbb{H}^d$.

 Let $a\in \Gamma \eta_i$ and  $D_a(\Gamma)$ be the Dirichlet region (see \eqref{eq:dirichelt}). 
Recall that $D_a(\Gamma)$ are convex compact polyhedra in $\H^d$, and that 
the set of the Dirichlet regions $D_a$, for all $a\in\Gamma \eta_i$, is a locally finite tessellation of $\mathbb{H}^d$. 
Using \eqref{eq:hyp dist}, the Dirichlet region 
can be written
$$D_a(\Gamma)=\{p\in\mathbb{H}^d\vert \langle a,p\rangle_-\geq\langle \gamma a,p\rangle_-, \forall \gamma\in\Gamma\setminus\{Id\} \}. $$
Let $a_1,a_2\in \Gamma \eta_i$ such that $D_{a_1}(\Gamma)$ and $D_{a_2}(\Gamma)$ have a common facet.
This facet is contained in the intersection of $\H^d$ with the hyperplane 
 $$\{p\in\R^{d+1}|\langle a_1,p\rangle_-=\langle a_2,p\rangle_-\},$$
and this hyperplane also contains $a_1^{\bot}\cap a_2^{\bot}$ by \eqref{eq def affine}.
It follows that vertices of $P_i$ (codimension $(d+1)$ faces)
 project along rays from the origin  onto the vertices of the Dirichlet tessellation.
In particular the vertices are in $\F$, 
so  $P_i\subset \F$, because it is the convex hull of its vertices \cite[1.4.3]{Sch93} and 
$\F$ is convex. In particular $P_i$ is a $\Gamma$-convex body due to Definition~\ref{def: fuchsian body}.
And  codimension $k$ faces of $P_i$ 
 projects onto codimension $k$ faces of the Dirichlet tessellation, so
 $P_i$ is locally finite with a countable number of facets.

Facets of $P_i$ are closed, as they project onto compact sets. 
In particular they are bounded  as contained in $\mathcal{F}$ hence compact. They are convex polytopes by construction, and Euclidean as contained in space-like planes.
Higher codimension faces are convex Euclidean polytopes as intersections of convex Euclidean polytopes.
\end{proof}

\paragraph{Support numbers}

The extended support function of a $\Gamma$-convex polyhedron $P$ is piecewise linear 
(it is linear on each solid angle determined by the normals
of the support planes at a vertex), it is why the data of the extended support function on each 
inward unit normal of 
the facets suffices to determine it. If $\eta_i$ is such a vector and $h$ is the support function of $P$, we call the positive 
number
$$h(i):=-h(\eta_i)$$ the \emph{$i$th support number} of $P$. 

 The facet with normal $\eta_i$ is denoted by $F_i$. Two adjacent facets
$F_i$ and $F_j$ meet at a codimension $2$ face  $F_{ij}$.
If three facets $F_i,F_j,F_k$ meet at a codimension $3$ face, then this face is denoted
by $F_{ijk}$.  We denote by $\varphi_{ij}$
the hyperbolic distance between  $\eta_i$ and $\eta_j$, given by (see for example 
\cite[(3.2.2)]{Rat06})
\begin{equation}\label{eq:hyp dist}
-\cosh \varphi_{ij}=\langle \eta_i,\eta_j\rangle_-. 
\end{equation}

Let $p_i$ be the foot of the perpendicular from the origin to the hyperplane $\mathcal{H}_i$ containing the facet $F_i$.  
In $\mathcal{H}_i$, let $p_{ij}$ be the foot of the perpendicular from $p_i$ to  $F_{ij}$. We denote by $h_{ij}$  
the signed distance
from $p_i$ to $p_{ij}$: it is non negative if $p_i$ is in the same side of $F_{j}$ than $P$.
See Figure~\ref{fig:suppnumb}.

\begin{figure}[ht]
\begin{center}
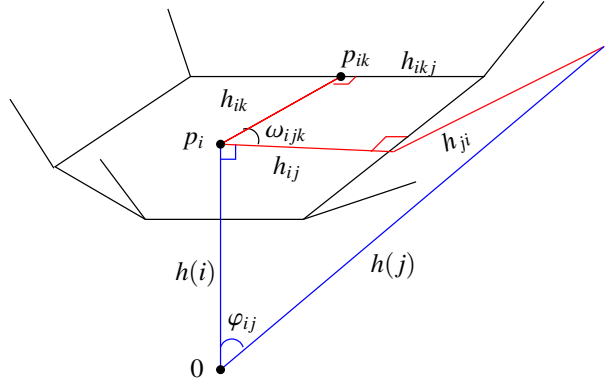
\end{center}
\caption{Supports numbers of a $\Gamma$-convex polyhedron.}\label{fig:suppnumb}
\end{figure}

For each $i$, $h_{ij}$ are the support numbers of the convex Euclidean polytope $F_i$.
($\mathcal{H}_i$ is identified with the Euclidean space $\R^d$, with $p_i$ as the origin.)
If we denote  by $\omega_{ijk}$ the angle between $p_ip_{ij}$ and $p_ip_{ik}$, it is well-known that \cite[(5.1.3)]{Sch93}
\begin{equation}\label{eq:supp nb eucl}
h_{ikj}=\frac{h_{ij}-h_{ik}\cos\omega_{ijk}}{\sin \omega_{ijk}}.
\end{equation}
We have a similar formula in Minkowski space \cite[Lemma~2.2]{polymink}:
\begin{equation}\label{eq: supp num mink}
 h_{ij}=-\frac{h(j)-h(i)\cosh \varphi_{ij}}{\sinh \varphi_{ij}}.
\end{equation}
In particular,
\begin{eqnarray}
 \ \label{eq:der lor1} &&\frac{\partial h_{ij}}{\partial h(j)}=-\frac{1}{\sinh \varphi_{ij}}, \\ 
\ &&\frac{\partial h_{ij}}{\partial h(i)}=\frac{\cosh \varphi_{ij}}{\sinh \varphi_{ij}}.\label{eq:der lor2}
\end{eqnarray}
If $h(i)=h(j)$ and if the quadrilateral is deformed under this condition, then
\begin{equation}
 \frac{\partial h_{ij}}{\partial h(i)}=\frac{\cosh \varphi_{ij}-1}{\sinh \varphi_{ij}}.\label{eq:der lor3}
\end{equation}

\paragraph{Space of polyhedra with parallel facets}

Let $P$ be a $\Gamma$-convex polyhedron. We label 
the facets of $P$
in a fundamental domain for the action of $\Gamma$. This set of label
is denoted by $\I$, and $\Gamma \I$ labels all the facets of $P$.
Let $R=(\eta_1,\ldots,\eta_n)$ be the inward unit normals of the facets of $P$ 
labeled by $\I$.

We denote by  $\mathcal{P}(\Gamma,R)$ the set of
$\Gamma$-convex  polyhedra with inward unit normals belonging to the set $R$.
By identifying a $\Gamma$-convex polyhedron with its  support numbers labeled by $\I$,  $\mathcal{P}(\Gamma,R)$
is a subset of $\mathbb{R}^n$. (The corresponding vector of $\R^n$ is the
 \emph{support vector} of the polyhedron.)
Note that this identification does not commute with the sum.
Because the sum of two piecewise linear functions is a piecewise linear function,
the Minkowski sum of two $\Gamma$-convex polyhedra is a $\Gamma$-convex polyhedron.
(More precisely, the linear functions under consideration are of the form $\langle \cdot, v\rangle_-$, 
with $v$ a vertex of a polyhedron, hence a future time-like vector, and the sum of two future
time-like vectors is a future time-like vector.)
But even if the two polyhedra have parallel facets, new facets can appear in the sum.
Later we will introduce a class of polyhedra such that the support vector of the Minkowski sum is 
the sum of the support vectors.

\begin{lemma}\label{lem:ens pol}
 The set $\mathcal{P}(\Gamma,R)$ is a non-empty open convex cone of $\mathbb{R}^n$.
\end{lemma}

\begin{proof}
The condition that the hyperplane supported by $\eta_j$ contains a facet of 
the polyhedron with support vector $h$ can be written as 
$$\exists x\in\mathbb{R}^{d+1}, \forall i\in\Gamma\mathcal{I}, i\not=j, \langle \eta_i,x\rangle_- <-h(i)\mbox{ and } 
\langle \eta_j,x\rangle_-= -h(j).$$
By \eqref{eq:hyp dist}  $\mathcal{P}(\Gamma,R)$ always contains the vector $(1,\ldots,1)$.
The set is  clearly open as a facet can't disappear for any sufficiently small deformation.
It is also clearly invariant under homotheties of positive scale factor. So to prove that $\mathcal{P}(\Gamma,R)$
is a convex cone it suffices to check that if $h$ and $h'$ belongs to 
$\mathcal{P}(\Gamma,R)$ then  $h+h'$ belongs to $\mathcal{P}(\Gamma,R)$.
It is immediate from the above characterization.
\end{proof}

\subsection{Covolume of convex Fuchsian polyhedra}

Let $F$ be a facet of a $\Gamma$-convex polyhedron $P$, contained in a space-like
hyperplane $\mathcal{H}$, with support number $h$. For the induced metric,
 $\mathcal{H}$ is isometric to the Euclidean space $\R^d$, in which
$F$ is a convex polytope, with volume $A(F)$. We call $A(F)$ the \emph{area} of the facet.
Let $C$ be the cone in $\R^{d+1}$ over $P$ with apex the origin. Its volume 
$V(C)$ is invariant under the action of an orientation and time-orientation preserving linear isometry
(they have determinant $1$), hence to compute $V(C)$ we can suppose that 
$\mathcal{H}$ 
is an horizontal hyperplane (constant last coordinate).
For horizontal hyperplanes, the induced metric is the same
if the ambient space is endowed with the standard Lorentzian metric or with the standard Euclidean metric.
So the well-known formula applies:
$$
V(C)=\frac{1}{d+1}h A(F),
$$
and then
$$
\mathrm{covol}(P)=\frac{1}{d+1}\sum_{i\in\mathcal{I}} h(i) A(F_i).
$$

Identifying $P$ with its support vector $h$, if $\big\langle \cdot,\cdot \big\rangle$ is the usual inner product of $\R^n$, 
we have
\begin{equation}\label{eq:def vol pol}
\mathrm{covol}(h)=\frac{1}{d+1}\big\langle h, A(h)\big\rangle \end{equation}
where $A(h)$ is the vector formed by the area of the facets $A(F_i)$.

\begin{lemma}\label{lem: der vol pom}
The function $\mathrm{covol}$ is $C^2$  on  $\R^n$, and for $h\in \mathcal{P}(\Gamma,R), X,Y\in \R^n$,
we have:
\begin{eqnarray}
 \  D_{h}\mathrm{covol}(X)=\big\langle X,A(h) \big\rangle, \label{eq: der vol pol}\\
\ D_{h}^2 \mathrm{covol} (X,Y)=\big\langle X, D_{h}A(Y)\big\rangle. \label{eq: der sec vol pol}
\end{eqnarray}
Moreover  \eqref{eq: der vol pol} is equivalent to  
\begin{equation}\label{eq: A self adj}
 \big\langle X, D_hA (Y)\big\rangle= \big\langle Y, D_hA (X)\big\rangle.
\end{equation}
\end{lemma}
\begin{proof}

Let $P$ be the polyhedron with support function $h\in \mathcal{P}(\Gamma,R)$.
Let $F_i$ be a facet of $P$, with support numbers $h_{i1},\ldots,h_{im}$. If $V_E$ is the $d$ Euclidean volume, it is well-known
that \cite[8.2.3]{Ale05}

\begin{equation}\label{eq:der area}
\frac{\partial V_E(F_i)}{\partial h_{ik}}=L_{ik}
\end{equation}
where $L_{ik}$ is the area of the facet of $F_i$ with support number $h_{ik}$ (for $d=1$, one has$1$
instead of $L_{ik}$).
$A(F_i)$ is not exactly as $V_E(F_i)$, because it is a function of $h$, and, when varying a $h(j)$, a new facet of $F_i$ can appear, 
as well as a new support number $h_{ij}$
of $F_i$. Actually many new facets can appear, as many as hyperplanes with normals $\Gamma \eta_j$ meeting
$F_i$. One has to consider $F_i$ as also supported by $h_{ij}$ (and eventually some orbits).
In this case, $L_{ij}=0$, and the variation of the volume is still given by formula \eqref{eq:der area}. 
So even if the combinatorics of $P$ changes under small change of a support number, there is no contribution
to the change of the volume of the facets. So \eqref{eq:der area}
gives 
\begin{equation}\label{eq:der areab}
\frac{\partial A(F_i)}{\partial h_{ik}}=L_{ik}.
\end{equation}

We denote by $E_i^j\subset\Gamma\mathcal{I}$ is the set of indices $k\in\Gamma j$ such that $F_k$ 
 is adjacent to $F_i$ along a codimension $2$ face. It can be empty. But for example 
if $\mathcal{I}$ is reduced to a single element $i$, $E_i^i$ is the set of facets adjacent to $F_i$
along a codimension $2$ face.
If $j\in\mathcal{I}\setminus\{i\}$ we get
$$\frac{\partial A(F_i)}{\partial h(j)}=\sum_{k\in E_i^j}\frac{\partial A(F_i)}{\partial h_{ik}}\frac{\partial h_{ik}}{\partial h(j)}$$
From \eqref{eq:der lor1} and \eqref{eq:der areab} it follows that 
\begin{equation}\label{eq: der par A}
\frac{\partial A(F_i)}{\partial h(j)}=-\sum_{k\in E_i^j}\frac{L_{ik}}{\sinh \varphi_{ik}}.
\end{equation}
For the diagonal terms:
\begin{eqnarray}\label{eq: coefdiag}
\ \frac{\partial A(F_i)}{\partial h(i)}&=&\sum_{j\in\mathcal{I}\setminus\{i\}}\sum_{k\in E_i^j} \frac{\partial A(F_i)}{\partial h_{ik}}\frac{\partial h_{ik}}{\partial h(i)}
+\sum_{k\in E_i^i}
 \frac{\partial A(F_i)}{\partial h_{ik}}\frac{\partial h_{ik}}{\partial h(i)}\nonumber \\
\ &\stackrel{(\ref{eq:der areab},\ref{eq:der lor2},\ref{eq:der lor3})}{=}&\sum_{j\in\mathcal{I}\setminus\{i\}}
\sum_{k\in E_i^j}\cosh \varphi_{ik}\frac{ L_{ik}}{\sinh \varphi_{ik}} +
 \sum_{k\in E_i^i} L_{ik}\frac{\cosh \varphi_{ik}-1}{\sinh \varphi_{ik}}.
\end{eqnarray} 
These expressions are continuous with respect to $h$, even if the combinatorics changes. So $A$ is $C^1$ and from
\eqref{eq:def vol pol} $\mathrm{covol}$ is $C^2$.

If \eqref{eq: der vol pol} is true, we get \eqref{eq: der sec vol pol}, and this expression
is symmetric as $\mathrm{covol}$ is $C^2$, so \eqref{eq: A self adj} holds.
Let us suppose that \eqref{eq: A self adj} is true. As made of volumes of convex polytopes of $\R^d$, $A$
 is homogeneous of degree $d$ so by 
Euler homogeneous theorem $D_hA (h)=dA(h)$. Using this
in \eqref{eq: A self adj} with $Y=h$ gives
$d\big\langle X, A (h)\big\rangle= \big\langle h, D_hA(X)\big\rangle.$
Now differentiating \eqref{eq:def vol pol} gives
$D_{h}\mathrm{covol}(X)=\frac{1}{d+1}\big\langle X,A(h) \big\rangle+\frac{1}{d+1}\big\langle h,D_hA(X) \big\rangle$. Inserting
the preceding equation leads to $\eqref{eq: der vol pol}$.

Let us prove \eqref{eq: A self adj}. If $e_1,\ldots,e_n$ is the standard basis of $\R^n$, it suffices to prove 
\eqref{eq: A self adj}
for $X=e_i$ and $Y=e_j$, $i\not= j$ i.e~that the gradient of $A$ is symmetric.
The sum in \eqref{eq: der par A} means that, in $\partial P/\Gamma$, each times the $i$th polytope meets the $j$th polytope along a codimension $2$
face, we add the quantity $\frac{L_{ik}}{\sinh \varphi_{ik}}$, which is symmetric in its arguments. 
Hence the gradient of $A$ is symmetric.
\end{proof}

Let us consider the simplest case of $\Gamma$-convex polyhedra in the Minkowski plane, 
with only one support number $h\in\R$.
Then by \eqref{eq: supp num mink}
$\mathrm{covol}(h)$ is equal to $h^2$ times a positive number, in particular it is  
a strictly convex function.
This is always true.

\begin{theorem}\label{thm: hess pos}
The Hessian of $\mathrm{covol}: \R^n \rightarrow \R$ is positive definite. 
\end{theorem}

Recall that we are looking at the covolume  on a space of support vectors, and not on
a space of polyhedra (the sum is not the same). 

\begin{proof}
Due to \eqref{eq: der sec vol pol} it suffices to study the Jacobian of $A$.
The elements off the diagonal are non-positive due to
\eqref{eq: der par A}. Note that the formula is also correct if  $E_i^j$ is empty. 
The diagonal terms \eqref{eq: coefdiag} are positive, as any facet $F_i$ has an adjacent facet. 
As $\cosh x>1$ for $x\not= 0$, \eqref{eq: coefdiag} and \eqref{eq: der par A} lead to
$$\frac{\partial A(F_i)}{\partial h(i)}  > \sum_{j\in\mathcal{I}\setminus\{i\}} \left|\frac{\partial A(F_i)}{\partial h(j)} \right|>0$$
that means that the Jacobian is strictly diagonally dominant
with positive diagonal entries, hence positive definite, see for example  \cite[1.22]{Var00}. 
\end{proof}

\subsection{Polyhedral Minkowski Theorem}\label{sec: mink}

We use a classical continuity method, although its Euclidean analog is more
often proved using a variational method.

\begin{theorem}[Minkowski  Theorem]\label{thm:minkex}
Let $\Gamma$ be a Fuchsian group, 
 $R=(\eta_1,\ldots,\eta_n)$ be a set of pairwise non collinear unit future time-like vectors 
 of the Minkowski space contained in
a fundamental domain of $\Gamma$,  and let $(f_1,\ldots,f_n)$ be  positive real numbers.

 There exists a unique $\Gamma$-convex polyhedron with inward unit normals $\eta_i$ such that the facet
orthogonal to $\eta_i$ has area $f_i$.
\end{theorem}

Theorem~\ref{thm:minkex} is equivalent to say that the map $\Phi$ from 
$\mathcal{P}(\Gamma,R)$ to $(\mathbb{R}_+)^n$ which associates to each
$(h_1,\ldots,h_n)\in \mathcal{P}(\Gamma,R)$ the  facet areas $(A(F_1),\ldots,A(F_n))$   is a bijection.
By Lemma~\ref{lem: der vol pom}, Theorem~\ref{thm: hess pos} and  
 local inverse theorem,   $\Phi$  is  locally invertible.
So $\Phi$ is a local homeomorphism by the invariance of domain theorem.
Lemma~\ref{prop:proprete} below says that $\Phi$ is proper. As $(\mathbb{R}_+)^n$ is connected,
it follows that $\Phi$ is surjective, hence a covering map.
But the target space $(\mathbb{R}_+)^n$ is  simply connected and $\mathcal{P}(\Gamma,R)$
is connected (Lemma~\ref{lem:ens pol}), so $\Phi$ is a homeomorphism, in particular bijective, and Theorem~\ref{thm:minkex} is proved.

\begin{lemma}\label{prop:proprete}
 The map $\Phi$ is proper: Let $(a_\alpha)_{\alpha\in\mathbb{N}}$  be a converging sequence of  $(\mathbb{R}_+)^n$
such that for all $\alpha$, there exists $h_\alpha=(h_{\alpha}(1),\ldots,h_{\alpha}(n))\in \mathcal{P}(\Gamma,R)$ with 
$\Phi(h_\alpha)=a_\alpha$. Then
 a subsequence of $(h_\alpha)_\alpha$ converges in $\mathcal{P}(\Gamma,R)$.
\end{lemma}
\begin{proof}
Let $\alpha\in\N$ and suppose that $h_{\alpha}(i)$ is the largest component of $h_{\alpha}$.
For any  support number $h_{\alpha}(j)$, $j\in\Gamma\mathcal{I}$,  of a facet adjacent to the one supported by 
$h_{\alpha}(i)$, as
$h_{\alpha}(i)\geq h_{\alpha}(j)$,
\eqref{eq: supp num mink} gives:
$$h_{ij}^{\alpha}=\frac{h_{\alpha}(i)\cosh \varphi_{ij}-h_{\alpha}(j)}{\sinh \varphi_{ij}}\geq h_{\alpha}(i)
 \frac{\cosh\varphi_{ij}-1}{\sinh \varphi_{ij}}.$$

As 
$\Gamma$ acts cocompactly on $\H^d$, for any $j\in\Gamma\mathcal{I}$,  
$\varphi_{ij}$ is bounded from below by a positive constant. 
Moreover the function
$x\mapsto  \frac{\cosh x-1}{\sinh x} $ is increasing, then there exists a positive number $\lambda_i$, depending only on $i$, such that
$$h_{ij}^{\alpha}\geq h_{\alpha}(i) \lambda_i.$$ 

As the sequence of areas of the facets is supposed to converge,
there exists positive numbers $A^+_i$ and  $A^-_i$
such that   $A^+_i\geq A(F_i^{\alpha}) \geq A^-_i$, where
$A(F_i^{\alpha})$ is the area  of the facet $F_i^{\alpha}$ supported by $h_{\alpha}(i)$. 
If $\mbox{Per}_i^{\alpha}$ (resp. $\mbox{Per}_i$) is the 
Euclidean $(d-1)$ volume of the hypersphere bounding the ball  with Euclidean $d$ volume $A(F_i^{\alpha})$ (resp. $A_i^-$), 
the isoperimetric inequality gives \cite[10.1]{BZ88}
$$\sum_j L_{ij}^{\alpha} \geq  \mbox{Per}_i^{\alpha} \geq \mbox{Per}_i,$$
where the sum is on the facets adjacent to $F_i^{\alpha}$ and $L_{ij}^{\alpha}$ is the $(d-1)$ volume of the codimension $2$ face between 
$F_i^{\alpha}$ and $F_j^{\alpha}$.
We get
$$A^+_i \geq A(F_i^{\alpha})=\frac{1}{d}\sum_j h_{ij}^{\alpha}L_{ij}^{\alpha} \geq h_{\alpha}(i) \lambda_i   
\frac{1}{d}\sum_j L_{ij}^{\alpha}  \geq h_{\alpha}(i) \frac{\lambda_i  \mbox{Per}_i}{d}.$$

As $h_{\alpha}(i)$ is the largest component of $h_{\alpha}$, all the support numbers are bounded
from above by a constant which does not depend on $\alpha$.
Moreover each component of $h_{\alpha}$ is positive, hence all the components of the elements of 
the sequence $(h_{\alpha})_{\alpha}$ are bounded from above and below,
so there exists a  subsequence $(h_{\phi(\alpha)})_{\phi(\alpha)}$ converging to $(h(1),\ldots,h(n))$, where $h(i)$ is
a non-negative number. 

Suppose that the limit of $(h_{\phi(\alpha)}(i))_{\phi(\alpha)}$ is zero. 
 Let $h_{\phi(\alpha)}(j)$ be the support number of a facet adjacent to $F_i^{\phi(\alpha)}$. 
If $\phi(\alpha)$ is sufficiently large, $h_{\phi(\alpha)}(j)$ is arbitrary close to $h(j)$, which is a non-negative number,
 and $h_{\phi(\alpha)}(i)$ is arbitrary close to $0$. By \eqref{eq: supp num mink},
$h_{ij}^{\alpha}$ is a non-positive number. So all the support numbers of $F_i^{\phi(\alpha)}$ are non-positive, 
hence  the $d$ volume of $F_i^{\phi(\alpha)}$ is non-positive, that
is impossible.
It follows easily that $(h_{\phi(\alpha)}(i))_{\phi(\alpha)}$ converges in $\mathcal{P}(\Gamma,R)$.
\end{proof}

\subsection{Mixed face area and mixed-covolume}

 Let us recall some basic facts about convex polytopes in Euclidean space (with non empty interior).
A convex polytope of $\mathbb{R}^d$ is  \emph{simple} 
if each vertex is contained in exactly $d$ facets. Each face of a simple convex polytope is 
a simple convex polytope. 
The \emph{normal fan} of a convex polytope is the decomposition of $\mathbb{R}^d$
by convex cones defined by the outward unit normals to the facets of the polytope
(each cone corresponds to one vertex).
Two convex polytopes are  \emph{strongly isomorphic}
if they have the same normal fan.
 The Minkowski sum of two  strongly isomorphic simple polytopes 
is a simple polytope strongly isomorphic to the previous ones.
Moreover the support vector of the Minkowski sum is the sum of the support vectors.

Let $Q$ be a simple convex polytope 
in $\R^d$ with $n$ facets. The set of convex polytopes of $\R^d$ strongly isomorphic 
to $Q$
 is a convex open cone in $\R^n$. The Euclidean volume $V_E$ is a 
polynomial of degree $d$ on this set, and its polarization 
$V_E(\cdot,\ldots,\cdot)$ is  the \emph{mixed-volume}.
The coefficients of the volume depend on the combinatorics, it's 
why we have to restrict ourselves to simple strongly isomorphic polytopes.
The following result is an equivalent formulation of the Alexandrov--Fenchel 
inequality.

\begin{theorem}[{\cite{Ale96,Sch93}}]\label{thm: AF eucl}
 Let $Q,Q_3,\ldots,Q_d$ be strongly isomorphic simple convex polytopes of $\R^d$ with $n$ facets
 and $Z\in\R^n$. Then
$$V_E(Z,Q,Q_3,\ldots,Q_d)=0  \Rightarrow V_E(Z,Z,Q_3,\ldots,Q_d)\leq 0$$
 and equality holds
if and only if $Z$ is the support vector of a point.
\end{theorem}

We identify a support  hyperplane
of an element of  $\mathcal{P}(\Gamma,R)$
with the Euclidean space $\R^d$ by performing a translation
along the ray from the origin orthogonal to the hyperplane.
In this way we consider all facets of elements of $\mathcal{P}(\Gamma,R)$
lying in parallel hyperplanes
as convex polytopes in the same Euclidean space $\R^d$.

The definition of  strong isomorphy and simplicity extend to  $\Gamma$-convex polyhedra, 
considering them as polyhedral hypersurface in the ambient vector space. 
Note that the simplest examples of Euclidean convex polytopes, the simplices, are simple, 
but the  simplest examples of $\Gamma$-convex polyhedra, those defined by only one orbit, 
are not simple (if $d>1$).
Let us formalize the definition of strong isomorphy. The \emph{normal cone} $N(P)$ of 
a convex $\Gamma$-polyhedron $P$ is the decomposition of $\F$ by convex cones defined by the 
inward normals to the facets of $P$.  It is the minimal decomposition
of $\F$ such that the extended support functions of $P$ is the restriction of a linear form on 
each part.
 If the normal fan $N(Q)$ subdivides $N(P)$,
then we write $N(Q) > N(P)$. Note that 
$$N(P+Q)>N(P). $$
Two convex $\Gamma$-polyhedron $P$ and $Q$ are \emph{strongly isomorphic} if $N(P)=N(Q)$.
If $P$ is simple, we denote by $[P]$  the subset of $\mathcal{P}(\Gamma,R)$ 
made of  polyhedra strongly isomorphic to $P$.

\begin{lemma}\label{lem:[P]}
All elements of $[P]$ are simple and $[P]$ is an open convex cone
of $\R^n$.
\end{lemma}
\begin{proof}
The fact that all elements of $[P]$ are simple and that $[P]$ is open
are classical, see for example \cite{Ale37}. The only difference with 
the Euclidean convex polytopes case is that,
around a vertex, two facets can belong to the same orbit for the action of $\Gamma$,
hence when one wants to slightly move a facet adjacent to a vertex, one actually moves two (or more) facets.
But this does not break the simplicity, nor the strong isomorphy class. 
Moreover $[P]$ is a convex cone as the sum of two functions piecewise linear on the same decomposition of 
$\F$ gives a piecewise linear function on the same decomposition. 
\end{proof}

Suppose that $P$ is simple, has $n$ facets (in a fundamental domain),
 and let $h_1,\ldots,h_{d+1}\in[P]$ (support vectors of polyhedra strongly isomorphic to $P$). 
Let us denote by $F_k(i)$ the $i$th facet of the polyhedron with support vector $h_k$, and let
 $h(F_k(i))$ be its support vector ($F_k(i)$ is seen as a convex polytope in $\R^d$).
The entries of  $h(F_k(i))$ have the form \eqref{eq: supp num mink} so the map $h_k\mapsto h(F_k(i))$ is linear.
This map can be defined formally  for all $Z\in\R^n$ using \eqref{eq: supp num mink}. 
The \emph{mixed face area}
$A(h_2,\ldots,h_{d+1})$ is the vector formed by the entries $V_E(h(F_2(i)),\ldots,h(F_{d+1}(i)))$, $i=1,\ldots,n$. 
Together with \eqref{eq:def vol pol}, this implies that $\mathrm{covol}$ is a $(d+1)$-homogeneous polynomial, and we call 
 \emph{mixed-covolume}
its polarization
$\mathrm{covol}(\cdot,\ldots,\cdot)$. Note that $\mathrm{covol}$ is $C^{\infty}$ on $[P]$.

\begin{lemma}\label{lem: gen mixed pol}
We have the following equalities, for $X_i\in\R^n$.
\begin{enumerate}[nolistsep,label={\bf(\roman{*})}, ref={\bf(\roman{*})}]
 \item $D^{d-1}_{X_2}A(X_3,\ldots,X_{d+1})=d! A(X_2,\ldots,X_{d+1})$,
\item $D_{X_1}\mathrm{covol} (X_2)=(d+1)\mathrm{covol}(X_2,X_1,\ldots,X_1)$,
\item $D^2_{X_1}\mathrm{covol} (X_2,X_3)=(d+1)d\mathrm{covol}(X_2,X_3,X_1,\ldots,X_1)$, \label{hess mixed pol}
\item $D^{d}_{X_1} \mathrm{covol} (X_2,\ldots,X_{d+1})=(d+1)!\mathrm{covol}(X_1,\ldots,X_{d+1})$,
\item $\mathrm{covol}(X_1,\ldots,X_{d+1})=\frac{1}{d+1}\big\langle X_1, A(X_2,\ldots,X_{d+1})\big\rangle$.
\label{mv pol}
\end{enumerate}
\end{lemma}
\begin{proof}
 The proof is analogous to the one of Lemma~\ref{eq: reg mix vol gen}.
\end{proof}

\begin{corollary}\label{cor: pol mix vol pos}
 For $h_i\in [P]$, $\mathrm{covol}(h_1,\ldots,h_{d+1})$ is non-negative.
\end{corollary}
\begin{proof}
As $h_i$ are support vectors of  strongly isomorphic
simple polyhedra, the entries of $A(h_2,\ldots,h_{d+1})$ are mixed-volume
of simple strongly isomorphic Euclidean convex polytopes, hence are non-negative
(see Theorem~5.1.6 in
 \cite{Sch93}).
The result follows from \ref{mv pol} because the entries of $h_1$ are positive.
\end{proof}

\begin{lemma}\label{lem: noyau trivial pol}
 For any $h_1,\ldots,h_{d-1}\in[P]$,   the symmetric bilinear form
$$\mathrm{covol}(\cdot,\cdot,h_1,\ldots,h_{d-1}) $$
has trivial kernel.
\end{lemma}
\begin{proof}
 The analog of the proof of Lemma~\ref{lem: noyau trivial reg}, using Theorem~\ref{thm: AF eucl}
instead of Theorem~\ref{thm:alg lin det}, gives that 
in each support hyperplane, the ``support vectors'' of $Z$ (formally given by  \eqref{eq: supp num mink})
are the ones of a point of $\R^d$. Let us denote by $Z_i$ the support vector of $Z$
in the hyperplane with normal $\eta_i$.
 
If $\epsilon$ is sufficiently small then $h_1+\epsilon Z$ is the support vector of a $\Gamma$-convex polyhedron
$P_1^{\epsilon}$ strongly isomorphic to 
$P_1$, the one with support vector $h_1$.
Moreover the support numbers of the $i$th facet $F_i$ of $P_1^{\epsilon}$
are the sum of the support numbers of the facet $F_i^1$ of $P_1$ with the coefficients of $\epsilon Z_i$. As  
$Z_i$ is the support vector of a point in $\mathbb{R}^d$, $F_i$ 
 is obtained form $F_i^1$ by a translation.
It follows that each facet of 
$P_1^{\epsilon}$ is obtained by a translation of the corresponding facet 
of $P_1$,  hence $P_1^{\epsilon}$ is a translate of $h_1$ (the translations of each facet have to coincide on each codimension
$2$ face). 
As $h_1+\epsilon Z$ is supposed to be a $\Gamma$-convex polyhedron for $\epsilon$ sufficiently
small, and as the translation of a $\Gamma$-convex polyhedron is not a $\Gamma$-convex polyhedron,
it follows that $Z=0$.
\end{proof}

\begin{theorem}\label{thm:hess vol def pos pol}
 For any $h_1,\ldots,h_{d-1}\in[P]$,   the symmetric bilinear form
$$\mathrm{covol}(\cdot,\cdot,h_1,\ldots,h_{d-1}) $$ is positive definite.
\end{theorem}
\begin{proof}
 The proof is analogous to the one of Theorem~\ref{thm:hess vol def pos reg}.
\end{proof}

\paragraph{Remark on spherical polyhedra}

The sets of strongly isomorphic simple $\Gamma$-convex polyhedra form convex cones in
 vector spaces (Lemma~\ref{lem:[P]}). The mixed-covolume allow to endow these vector spaces with an inner product. Hence, if
we restrict to polyhedra of covolume $1$, those sets are isometric to convex spherical polyhedra. 
For $d=1$ we get simplices named orthoschemes \cite{polymink}.
In $d=2$, if we look at the metric induced on the boundary of the Fuchsian polyhedra, 
we get spherical metrics on subsets of the spaces of flat metrics with cone-singularities of negative curvature on the 
compact surfaces of genus $>1$. It could be interesting to investigate the shape of these subsets.

\section{General case}\label{sec:gen}

\subsection{Convexity of the covolume}

\paragraph{Hausdorff metric}

Recall that $\K(\Gamma)$ is the set of $\Gamma$-convex bodies for a given $\Gamma$. For $K,K'$ we
define the \emph{Hausdorff metric} by
$$ d(K,K')=\min\{\lambda \geq 0 |  K'+\lambda B \subset K,   K+\lambda B\subset K' \}. $$

It is not hard to check that this is a distance and that  Minkowski sum and multiplication by a positive scalar are continuous
for this distance. 
If we identify $\Gamma$-convex bodies with their support functions, then $\K(\Gamma)$
is isometric to a convex cone in $C^0(\H^d/\Gamma)$ endowed with the maximum norm, i.e.:
$$d(K,K')=\sup_{\eta\in\H^d/\Gamma} |h(\eta)-h'(\eta)|.$$
The proofs is easy and formally the same as in the Euclidean case 
\cite[1.8.11]{Sch93}. 

\begin{lemma}
The covolume is a continuous function. 
\end{lemma}
\begin{proof}
Let $K$ be in $\K(\Gamma)$ with support function $h$.
For a given $\epsilon>0$, choose $\lambda>1$ such that $(\lambda^{d+1}-1)\lambda^{d+1} \mathrm{covol}(K)<\epsilon$. 
Let $\rho <0$ such that $h>\rho$, and let $\overline{\alpha}>0$ be the
minimum of $h-\rho$. Let $\alpha=\mbox{min}(\overline{\alpha},(1-\lambda)\rho)>0$. In particular, 
\begin{equation}\label{eq alpha}
\rho\leq h-\alpha. 
\end{equation}
Finally, let
$\overline{K}$ with support function $\overline{h}$ be such that
$d(K,\overline{K})<\alpha$. In particular, $ h-\alpha < \overline{h}$, 
that, inserted in \eqref{eq alpha}, gives that $\rho <   \overline{h}$. This and the definition of
$\alpha$ give
$$\overline{h} \leq h+\alpha \leq h+(1-\lambda)\rho\leq h+(1-\lambda)\overline{h},$$
i.e.~$\lambda \overline{h} \leq h$, i.e.~$\lambda \overline{K}\subset K$, in particular
$\mathrm{covol}(K)\leq \lambda^{d+1}\mathrm{covol}(\overline{K})$. In a similar way we get $\mathrm{covol}(\overline{K})\leq \lambda^{d+1}\mathrm{covol}(K)$.
This allows to write
\begin{eqnarray*}
 \ && \mathrm{covol}(K)-\mathrm{covol}(\overline{K}) \leq (\lambda^{d+1}-1)\mathrm{covol}(\overline{K})\leq (\lambda^{d+1}-1)\lambda^{d+1} \mathrm{covol}(K)<\epsilon \\
 \ && \mathrm{covol}(\overline{K})-\mathrm{covol}(K) \leq (\lambda^{d+1}-1)\mathrm{covol}(K)\leq (\lambda^{d+1}-1)\lambda^{d+1} \mathrm{covol}(K)<\epsilon 
\end{eqnarray*}
i.e.~$|\mathrm{covol}(K)-\mathrm{covol}(\overline{K})|< \epsilon$.
\end{proof}

The general results are based on polyhedral approximation.

\begin{lemma}\label{lem: approximation}
Let $K_1,\ldots,K_p\in \K(\Gamma)$. There exists 
a sequence $(P^1_k,\ldots,P^p_k)_k$ of strongly isomorphic simple $\Gamma$-convex polyhedra converging to 
$(K_1,\ldots,K_p)$. 
\end{lemma}
\begin{proof}
First, any $\Gamma$-convex body $K$ is arbitrarily close to a $\Gamma$-convex polyhedron $Q$.
Consider a finite number of points on $K$ and let $Q$ be the polyhedron made by the 
hyperplanes orthogonal to the orbits of these points, and passing through these points. We get $K\subset Q$.
For any $\epsilon >0$, if $Q+\epsilon B$ is not included in $K$
then add facets to $Q$. The process ends by cocompactness.

Let $Q^i$ be a $\Gamma$-convex polyhedron arbitrary close to $K_i$, and let $P$ be the $\Gamma$-convex polyhedron 
$Q^1+\cdots+Q^p$. Let us suppose that around a vertex $x$ of $P$, two facets belong
to the same orbit for the action of $\Gamma$.
We perform
 a little translation in direction of $P$ of a support hyperplane at $x$, which is not a support
hyperplane of a face containing $x$. A new facet appears, the vertex
$x$ disappears, and the 
two facets in the same orbit share one less vertex. Repeating this operation a finite number of times,
we get a polyhedron $P'$ with $N(P')>N(P)$ and such that around each vertex, no facets belong to the
same orbit.
If  $P'$ is not simple, there exists a vertex 
$x$ of $P'$ such that more than $d+1$ facets meet at this vertex. 
 We perform a small little
parallel move of one of this facets.
In this case the number of facets meeting at the vertex $x'$ corresponding to $x$ decreases, 
and new vertices can appear, but the number of facets meeting at each of those vertices is strictly less than the number
of facets meeting at $x$. If the move is sufficiently small, the number of facets meeting at the other vertices is not
greater than it was on $P'$.
Repeating this operation a finite number of times leads to the simple polyhedra $P''$, and $N(P'')>N(P')$.

Now we define $P^i=Q^i+\alpha P''$, with $\alpha>0$ sufficiently small such that
$P^i$ remains close to $Q^i$ and hence close to $K_i$. By definition of
$P$, $N(P)>N(Q^i)$ and finally $N(P'')>N(Q^i)$ hence $N(P^i)=N(P'')$: all the
$P^i$ are strongly isomorphic to $P''$, which is simple.
\end{proof}

\begin{theorem}\label{them: vol conv}
 The covolume is a convex function on the space of $\Gamma$-convex bodies: for any $K_1, K_2\in \K(\Gamma)$,
$\forall t\in[0,1]$,
$$\mathrm{covol}((1-t)K_1+tK_2)\leq t\mathrm{covol}(K_1)+(1-t)\mathrm{covol}(K_2).$$
\end{theorem}

\begin{proof}
By Lemma~\ref{lem: approximation}, there exist strongly isomorphic simple $\Gamma$-convex polyhedra 
$P_1$ and $P_2$ arbitrary close to respectively $K_1$ and $K_2$.  
 As for simple strongly isomorphic $\Gamma$-convex polyhedra,
the addition of support vectors is the same as Minkowski addition,
 Theorem~\ref{thm: hess pos} gives that
$$\mathrm{covol}((1-t)P_1+tP_2)\leq t\mathrm{covol}(P_1)+(1-t)\mathrm{covol}(P_2)$$
 and the theorem follows by continuity of the covolume.
\end{proof}

\subsection{Mixed covolume and standard inequalities}

\begin{lemma}
 The covolume on $\K(\Gamma)$ is a homogeneous polynomial of degree $(d+1)$. Its polar form is
the \emph{mixed-covolume} $\mathrm{covol}(\cdot,\ldots,\cdot)$, a  continuous non-negative symmetric map  on 
$(\K(\Gamma))^{d+1}$ such that
$$\mathrm{covol}(K,\ldots,K)=\mathrm{covol}(K).$$ 

Moreover if we restrict to a space of  strongly isomorphic simple $\Gamma$-convex polyhedra,
or to the space of $C^{\infty}_+$  $\Gamma$-convex bodies,  then 
$\mathrm{covol}(\cdot,\ldots,\cdot)$ is the same map as the one previously considered.
\end{lemma}
\begin{proof}
 Let us define
\begin{equation}\label{eq:polar}
\mathrm{covol}(K_1,\ldots,K_{d+1})=\frac{1}{(d+1)!}\sum_{i=1}^{d+1} (-1)^{d+1+k}\sum_{i_1<\cdots<i_{d+1}} \mathrm{covol}(K_{i_1}+\cdots +K_{i_{d+1}}) 
\end{equation}
which is a symmetric map. From the continuity of the covolume and of the Minkowski addition, it is a continuous map.
In the case when $K_i$ are  strongly isomorphic simple polyhedra,
the right-hand side of \eqref{eq:polar} to the mixed-covolumes previously introduced \cite[5.1.3]{Sch93}
(we also could have used another polarization formula \cite[(A.5)]{Hor07}). Let us consider a sequence of
 strongly isomorphic simple $\Gamma$-convex polyhedra 
$P_1(k),\ldots,P_{d+1}(k)$ converging to $K_1,\ldots,K_{d+1}$ (Lemma~\ref{lem: approximation}).
From the definition of the mixed-covolume we have
$$\mathrm{covol}(\lambda_1 P_1(k)+\cdots+\lambda_{d+1} P_{d+1}(k))=\sum_{i_1,\ldots,i_{d+1}=1}^{d+1} \lambda_{i_1}\cdots\lambda_{i_{d+1}}
\mathrm{covol}(P(k)_{i_1},\ldots,P(k)_{i_{d+1}})$$
and by continuity, passing to the limit,
$$\mathrm{covol}(\lambda_1 K_1+\cdots+\lambda_{d+1}  K_{d+1})=\sum_{i_1,\ldots,i_{d+1}=1}^{d+1} \lambda_{i_1}\cdots\lambda_{i_{d+1}}
\mathrm{covol}(K_{i_1},\ldots,K_{i_{d+1}})$$
so the covolume is a polynomial, and $\mathrm{covol}(\cdot,\ldots,\cdot)$ introduced at the beginning of the proof is its polarization.
It is non-negative due to Corollary~\ref{cor: pol mix vol pos}.

In the case of $C^{2}_+$  $\Gamma$-convex bodies, both notions of mixed-covolume satisfy  \eqref{eq:polar}.
\end{proof}

\begin{theorem}\label{thm:general}
Let $K_i\in \K(\Gamma)$ and $0<t<1$. We have the following inequalities.
\begin{eqnarray*}
\ &&\mbox{\emph{Reversed Alexandrov--Fenchel inequality:}} \\
\ &&\mathrm{covol}(K_1,K_2,K_3,\ldots,K_{d+1})^2\leq \mathrm{covol}(K_1,K_1,K_3,\ldots,K_{d+1})\mathrm{covol}(K_2,K_2,K_3,\ldots,K_{d+1}) \\ 
\ &&\mbox{\emph{First reversed Minkowski inequality:}}\\
\ &&\mathrm{covol}(K_1,K_2,\ldots,K_2)^{d+1}\leq \mathrm{covol}(K_2)^{d}\mathrm{covol}(K_1)\\
\ &&\mbox{\emph{Second or quadratic reversed Minkowski inequality:}}\\
\ &&\mathrm{covol}(K_1,K_2,\ldots,K_2)^2\leq \mathrm{covol}(K_2)\mathrm{covol}(K_1,K_1,K_2,\ldots,K_2)\\
\ &&\mbox{\emph{Reversed Brunn--Minkowski inequality:}} \\
\ &&\mathrm{covol}((1-t)K_1+tK_2)^{\frac{1}{d+1}} \leq (1-t)\mathrm{covol}(K_1)^{\frac{1}{d+1}}+t\mathrm{covol}(K_2)^{\frac{1}{d+1}} \\
\ &&\mbox{\emph{Reversed linearized first Minkowski inequality:}}\\
\ && (d+1) \mathrm{covol}(K_1,K_2,\ldots,K_2)\leq d\mathrm{covol}(K_2)+\mathrm{covol}(K_1) 
\end{eqnarray*}
If all the $K_i$ are  $C^{\infty}_+$  or  strongly isomorphic simple polyhedra, then equality holds 
in reversed Alexandrov--Fenchel and  second reversed Minkowski inequalities if and only if $K_1$ and $K_2$ are homothetic.
\end{theorem}

In the classical case of Euclidean convex bodies, the linearized first Minkowski inequality is valid only
on particular subsets of the space of convex bodies, see \cite[(6.7.11)]{Sch93}.

\begin{proof}
 Let $P_1(k),\ldots,P_{d+1}(k)$ be a sequence of simple strongly isomorphic $\Gamma$-convex polyhedra
  converging to $K_1,\ldots,K_{d+1}$ (Lemma~\ref{lem: approximation}).
Applying Cauchy--Schwarz inequality to the inner product
$\mathrm{covol}(\cdot,\cdot,P_3(k),\ldots,P_{d+1}(k))$ (Theorem~\ref{thm:hess vol def pos pol}) at
$(P_1(k),P_2(k))$ and passing to the limit gives reversed Alexandrov--Fenchel inequality.
Equalities cases follow from Theorem~\ref{thm:hess vol def pos reg} and
\ref{thm:hess vol def pos pol}. 
The second reversed Minkowski inequality and its equality case follows from Alexandrov--Fenchel inequality.

As the covolume is convex (Theorem~\ref{them: vol conv}), for $\overline{K}_1$ and $\overline{K}_2$ of unit
covolume, for $\overline{t}\in[0,1]$ we get
$$\mathrm{covol}((1-\overline{t})\overline{K}_1+\overline{t}\overline{K}_2)\leq 1.$$
Taking $\overline{K}_i=K_i/\mathrm{covol}(K_i)^{\frac{1}{d+1}}$ and
$$\overline{t}=\frac{t\mathrm{covol}(K_2)^{\frac{1}{d+1}}}{(1-t)\mathrm{covol}(K_1)^{\frac{1}{d+1}}+t\mathrm{covol}(K_2)^{\frac{1}{d+1}}} $$ 
leads to the reversed  Brunn--Minkowski inequality. 

As $\mathrm{covol}(\cdot)$ is convex, the map
$$f(\lambda)=\mathrm{covol}((1-\lambda)K_1+\lambda K_2)-(1-\lambda)\mathrm{covol}(K_1)-\lambda \mathrm{covol}(K_2), 0\leq \lambda \leq 1, $$
is convex. As $f(0)=f(1)=0$, we have $f'(0)\leq 0$, that is the reversed linearized first Minkowski inequality.
(Remember that
$$\mathrm{covol}((1-\lambda)K_1+\lambda K_2)=(1-\lambda)^{d+1}\mathrm{covol}(K_1)+(d+1)(1-\lambda)^d\lambda \mathrm{covol}(K_1,\ldots,K_1,K_2)+\lambda^2[\ldots].) $$

Reversed Brunn--Minkowski says that the map $\mathrm{covol}(\cdot)^{\frac{1}{d+1}}$ is convex. Doing the same as above with the convex map
$$g(\lambda)=\mathrm{covol}((1-\lambda)K_1+\lambda K_2)^{\frac{1}{d+1}}-(1-\lambda)\mathrm{covol}(K_1)^{\frac{1}{d+1}}-\lambda \mathrm{covol}(K_2)^{\frac{1}{d+1}}, 0\leq \lambda \leq 1, $$
leads to the first reversed Minkowski inequality.
\end{proof}

The \emph{(Minkowski) area} $S(K)$ of a $\Gamma$-convex body $K$ is $(d+1)\mathrm{covol}(B,K,\ldots,K)$. 
Note that it can be defined from the covolume:
$$S(K)=\lim_{\epsilon\rightarrow 0^+} \frac{\mathrm{covol}(K+\epsilon B)-\mathrm{covol}(K)}{\epsilon}. $$

The following  inequality says that, among $\Gamma$-convex bodies of 
area $1$, $B$ has smaller covolume, or equivalently that among $\Gamma$-convex bodies of 
covolume $1$, $B$ has larger area. 

\begin{corollary}[Isoperimetric inequality]
 Let $K$ be a $\Gamma$-convex body. Then
$$\left(\frac{S(K)}{S(B)}\right)^{d+1}\leq \left(\frac{\mathrm{covol}(K)}{\mathrm{covol}(B)}\right)^d. $$
\end{corollary}
\begin{proof}
 It follows from the  first reversed Minkowski with $K_1=B$, $K_2=K$, divided by
$S(B)^{d+1}$, with $(d+1)\mathrm{covol}(B)=S(B)$.
\end{proof}

\begin{lemma}
 If $K$ is a $C^{\infty}_+$ $\Gamma$-convex body, then $S(K)$ is the volume of the
Riemannian manifold $\partial K/\Gamma$.

If $K$ is a $\Gamma$-convex polyhedron, then $S(K)$ is the total face area of $K$ (the sum of the area of the facets of
$K$ in a fundamental domain).
\end{lemma}

In particular
$S(B)$ is the volume of the compact hyperbolic manifold $ \H^d/\Gamma$.

\begin{proof}
The $C^2_+$  case follows from the formulas in Section~\ref{sec:reg}, because
 $B$ is a $C^2_+$  convex body.

Let $K$ be polyhedral.  Let  $(P_k)_k$ be a sequence of polyhedra 
converging to $B$ and such that all the support numbers of $P_k$ are equal to $1$ (i.e.~all facets are tangent to $\H^d$).
Up to add facets, we can construct $P_k$ such that
$N(P_k)>N(K)$ and  $P_k$ is
simple. Let $\alpha$ be a small positive number. The polyhedron
$K+ \alpha P_k$ is strongly isomorphic to $P_k$.
It follows from formulas of 
Section~\ref{sec:pol} than $(d+1)\mathrm{covol}(P_k,K+ \alpha P_k,\ldots,K+ \alpha P_k)$
is equal to the total face area of $K+\alpha P_k$. By continuity of the mixed-covolume,
 $(d+1)\mathrm{covol}(P_k,K+ \alpha P_k,\ldots,K+ \alpha P_k)$ converges to $(d+1)\mathrm{covol}(P_k,K,\ldots,K)$
when $\alpha$ goes to $0$.

 We associate to $K$ a support vector
$h(K)$
whose entries are support numbers of facets of $K$, but also
to support hyperplanes of $K$ parallel to facets of $P_k$. We also consider
``false faces'' of larger codimension, such that the resulting normal fan is the same as the 
one of $P_k$.  This is possible
as the normal fan of $P_k$ is finer than the one of $K$. The support numbers
of the false faces can be computed using \eqref{eq:supp nb eucl} and \eqref{eq: supp num mink}
(i.e.~$K$ is seen as an element of the closure of $[P_k]$). In particular 
$h(K+\alpha P_k)=h(K)+\alpha h(P_k)$ and as the map $s$ giving the support numbers of a facet
in terms of the support numbers of the polyhedron is linear,
 the area of this facet  is
$V_E(s(h(K))+\alpha s(h(P_k)))$.  By continuity of the Euclidean volume, when $\alpha$ goes to
$0$ this area goes to the area of the facet of $K$ (it is $0$ if the facet was a ``false facet'' of $K$).
Hence $(d+1)\mathrm{covol}(P_k,K,\ldots,K)$ is equal to the total face area of $K$,
and on the other hand it goes to $S(K)$  when $k$ goes to infinity.
\end{proof}

Let us end with an example. Let $K$ be a polyhedral $\Gamma$-convex body with support numbers equal to $1$. 
In this case $S(K)=(d+1)\mathrm{covol}(K)$, and as $S(B)=(d+1)\mathrm{covol}(B)$, the isoperimetric inequality becomes
$$\frac{S(K)}{S(B)}\leq 1.$$
Let $d=2$ and $\Gamma$ be the Fuchsian group which has a regular octagon as a fundamental domain 
in the Klein model of $\H^2$. Then by the Gauss--Bonnet theorem $S(B)=4\pi$.
The total face area of $K$ is the area of only one facet, which is eight times the area of a Euclidean triangle of 
height $h'=\frac{\cosh \varphi-1}{\sinh \varphi}$ and with edge length two times
$h'\frac{1-\cos \pi/4}{\sin \pi/4}$ (see \eqref{eq:supp nb eucl} and \eqref{eq: supp num mink}). 
$\varphi$ is the distance between a point of $\H^2$ and its image by a generator of $\Gamma$,
and $\cosh  \varphi = 2+2\sqrt{2}$ (compare Example~C p.~95 in \cite{kat92} with Lemma~12.1.2 in \cite{MR03}).
By a direct computation the isoperimetric inequality becomes
$$0,27\approx 13-9\sqrt{2} \leq \frac{\pi}{2}\approx 1,57.$$

\paragraph{Remarks on equality cases and general Minkowski theorem}

Brunn--Minkowski inequality for non-degenerated (convex) bodies in the Euclidean space comes with a description of the
 equality case. Namely, the equality occurs for a $t$ if and only if
the bodies are homothetic (the part ``if'' is trivial). 
At a first sigh it is not possible to adapt the standard proof of the equality case to
the Fuchsian case,
as it heavily lies on translations \cite{BF87,Sch93,Ale05}. 

If such a result was known, it should imply, in a way formally equivalent to the classical one,
the characterization of the equality case
in the reversed first Minkowski inequality, as well as the uniqueness part in the Minkowski theorem 
and the equality case in the isoperimetric equality (see below).

The Minkowski problem in the classical case is 
to find a convex  body having a prescribed measure as ``area measure'' (see notes of Section~5.1 in \cite{Sch93}).
It can be solved by approximation
(by $C^2_+$ or polyhedral convex bodies), see \cite{Sch93}, or by a variational argument using the volume,
see \cite{Ale96}. Both methods require a compactness result, which is known  as 
the Blaschke selection Theorem.
Another classical question about Minkowski problem in the $C^2_+$ case, is to know the regularity of the hypersurface 
with respect
to the regularity of the curvature function, see the survey \cite{TW08}. 
All those questions can be transposed in the setting of Fuchsian convex bodies.

\begin{spacing}{0.9}
\begin{footnotesize}
\bibliographystyle{apalike}
\bibliography{mink}
\end{footnotesize}
\end{spacing}

\end{document}